\documentclass[10pt, reqno]{amsart}

\usepackage{amssymb}
\usepackage[pdftex]{graphicx}
\usepackage{epstopdf}
\usepackage{float}
\usepackage{amsmath}
\usepackage[margin=1.3in]{geometry}
\usepackage{amsthm}
\usepackage[all]{xy}
\usepackage{amsfonts}
\usepackage{txfonts}
\usepackage{mathrsfs}
\usepackage{mathdots}
\usepackage{mathpazo}
\usepackage{tikz}

\usepackage{latexsym}
\usepackage[shortlabels]{enumitem}
\usepackage{hyperref}
\usepackage{tikz-cd}
\newtheorem{theorem}{Theorem}[section]
\newtheorem{lemma}[theorem]{Lemma}
\newtheorem{proposition}[theorem]{Proposition}
\newtheorem{conjecture}[theorem]{Conjecture}

\theoremstyle{definition}
\newtheorem{definition}[theorem]{Definition}

\theoremstyle{remark}
\newtheorem{remark}[theorem]{Remark}

\numberwithin{equation}{section}

\makeatletter
\g@addto@macro\@floatboxreset\centering
\makeatother

\title{congruence relations of GSpin Shimura varieties}
\author{Hao Li}
\date{April 22, 2019.}

\begin{document}

\begin{abstract}
We prove the Chai-Faltings version of the Eichler-Shimura congruence relation for simple GSpin Shimura varieties with hyperspecial level structures at a prime $p$. 
\end{abstract}
\maketitle 

\section{Introduction}

Let $G$ be a reductive group over $\mathbb{Q}$, and $(G, X)$ be a Shimura datum associated to $G$ with reflex field $E$. Fix a prime $p$ over which $G$ is nonramified. Let $K=K^pK_p\subset G(\mathbb{A}_f)$ such that $K_p$ is a hyperspecial subgroup of $G(\mathbb{Q}_p)$ and $K^p$ is sufficiently small. There is a smooth complex manifold $G(\mathbb{Q})\backslash X\times G(\mathbb{A}_f)/K$ which admits a canonical algebraic variety structure defined over $E$ and denote it by $Sh_K$. Its base change to the algebraic closure of $E$, $Sh_K\times_{E}\overline{E}$ (as a scheme over $E$) admits a natural $Gal(\overline{E}/E)$ action, which induces a $Gal(\overline{E}/E)$ action on the cohomology. On the other hand, the Shimura varieties with deeper level structures at $p$ define a family of self correspondences of $Sh_K\times_{E}\overline{E}$, which generate a subalgebra of $Corr(Sh_K\times_{E}\overline{E}, \overline{\mathbb{Q}}_l)$, via the algebra following homomorphism
\begin{equation*}
\mathscr{H}(G(\mathbb{Q}_p)//K_p,\mathbb{Q})\longrightarrow Corr(Sh_K\times_{E}\overline{E},\overline{\mathbb{Q}}_l)
\end{equation*}
in which $\mathscr{H}(G(\mathbb{Q}_p)//K_p,\mathbb{Q})$ ($l$ is a prime different from $p$) is the spherical Hecke algebra at $p$. One of the major question in arithmetic geometry and automorphic forms is to understand the relation between these two actions. 

In\cite{BR}, Blasius and Rogawski constructed the Hecke polynomial $H(X)$, which is a polynomial with coefficients in $\mathscr{H}(G(\mathbb{Q}_p)//K_p,\mathbb{Q})$, and conjectured that
\begin{conjecture}
In the algebra $End_{\overline{\mathbb{Q}}_l}(H^{\bullet}(\overline{Sh_{K,\overline{E}}},IC^\bullet))$, the following equation holds
\begin{equation}
H(Fr_p) = 0
\end{equation}
in which $\overline{Sh_K}$ means the Baily-Borel compactification of the Shimura variety and $IC^\bullet$ is the intersection cohomology sheaf with middle perversity, and $Fr_p$ is the conjugacy class of the geometric Frobenius in $Gal(\overline{\mathbb{Q}}/E)$. 
\end{conjecture}

Their conjecture is a generalization of the celebrated Eichler-Shimura theorem which plays an important role in the development of the arithmetic theory of elliptic curves and modular forms. Also, there are several variants of the above conjecture. 

In\cite{FC}, Chai and Faltings formulated a cycle version of the congruence relation, which will be the main concern of this paper. The precise statement of this version of the conjecture was made by Koskivirta in\cite{K}. Let $Sh_K$ be a Kottwitz PEL-type Shimura variety attached to the Shimura datum $(G,X)$ with level structure $K=K_pK^p$ such that $G$ is nonoramified at $p$ and $K_p$ is hyperspecial. Let $E$ be its reflex field as above. Let $v$ be a prime of $E$ over $p$, $O_{E,(v)}$ its ring of integers localized at $v$ and $k$ its residue field. Let $\mathscr{S}_K$ be its semi-global integral model over $O_{E,(v)}$. This scheme is the moduli space of abelian schemes with certain additional structures. Moonen, following Chai-Faltings, defined a stack $p-Isog$ over $O_{E,(v)}$, classifying pairs $(f:A_1\longrightarrow A_2)$ of $p$-isogenies between two abelian schemes respecting all the additional structures. It has two natural projections to $\mathscr{S}_K$, assigning the pair its source and target. Then Moonen defined two algebras $\mathbb{Q}[p-Isog_E]$ and $\mathbb{Q}[p-Isog_k]$. These algebras can be thought as a geometric realization of the local spherical Hecke algebra $\mathscr{H}(G(\mathbb{Q}_p)//K_p,\mathbb{Q})$. Moonen also defined $p-Isog_k^{ord}$ as the preimage of the $\mu$-ordinary locus of the special fiber of the Shimura variety $\mathscr{S}_k$ under the source projection and an algebra $\mathbb{Q}[p-Isog_k^{ord}]$. The specialization map induces a linear map
\begin{equation*}
\sigma:\mathbb{Q}[p-Isog_E]\longrightarrow\mathbb{Q}[p-Isog_k]
\end{equation*}
There are two maps of schemes (stacks): intersection with the ordinary locus
\begin{equation*}
ord:p-Isog_k\longrightarrow p-Isog_k^{ord}
\end{equation*}
closure map
\begin{equation*}
cl:p-Isog_k^{ord}\longrightarrow p-Isog_k
\end{equation*}
These in turn induce corresponding maps on the algebra of $p$-isogenies. 

Moonen defined a distinguished element $F\in\mathbb{Q}[p-Isog_k]$, called the Frebenius cycle, as the image of a section of the source projection $s$, sending an abelian scheme $A$ to its relative Frobenius $A\longrightarrow A^{(p)}$. Then the Chai-Faltings' version of the conjecture is the following
\begin{conjecture}\label{conj}
Via the map
\begin{equation}
\sigma\circ h:\mathscr{H}(G(\mathbb{Q}_p)//K_p,\mathbb{Q})\longrightarrow\mathbb{Q}[p-Isog_k]
\end{equation}
view the Hecke polynomial $H(X)$ as a polynomial with coefficients in $\mathbb{Q}[p-Isog_k]$. Then the cycle $F$ as defined above, is a root of $H(X)$. 
\end{conjecture}

There is an "ordinary version" of the above conjecture. Via intersecting with the ordinary locus and taking the closure, one gets a map: $\sigma\circ ord:\mathbb{Q}[p-Isog_k]\longrightarrow\mathbb{Q}[p-Isog_k]$. Composing with $\sigma\circ h$, There is another map from the Hecke algebra $\mathscr{H}(G(\mathbb{Q}_p)//K_p,\mathbb{Q})$. Then the ordinary congruence relation reads
\begin{theorem}
Via the map
\begin{equation}
cl\circ ord\circ\sigma\circ h:\mathscr{H}(G(\mathbb{Q}_p)//K_p,\mathbb{Q})\longrightarrow\mathbb{Q}[p-Isog_k]
\end{equation}
view the Hecke polynomial as a polynomial with coefficients in $\mathbb{Q}[p-Isog_k]$. Then the cycle $F$ as defined above, is a root of $H(X)$. 
\end{theorem}

For the more precise definitions of the symbols appearing in the previous paragraph and the conjectures, see \ref{mspi} and \ref{twosections}. This conjecture was proposed implicitly in the last chapter of their book\cite{FC}. The conjecture for the PEL-type Shimura varieties was described by Moonen in\cite{M}, 4.2 in detail and made more explicitly by Koskivirta\cite{K}.

Conjecture \ref{conj} was verified by B\"ultel, Wedhorn, B\"ultel-Wedhorn and Koskivirta \cite{K} for some Shimura varieties. 

Their general strategy is as follows: Starting with the ordinary congruence relation, which is already proved, try to prove that the image of $\mathscr{H}(G(\mathbb{Q}_p)//K_p,\mathbb{Q})$ under $\sigma\circ h$ is contained in $cl\circ ord(\mathbb{Q}[p-Isog_k])$, so the ordinary version of the congruence relation implies the full version of the conjecture. However, this is not always true in general. In the case studied by Koskivirta, there are actually supersingular cycles in the image of $\mathscr{H}(G(\mathbb{Q}_p)//K_p,\mathbb{Q})$ under $\sigma\circ h$. He factored the Hecke polynomial so that a ``special factor" kills these supersingular cycles, thus the full version of the conjecture still holds. 

In the more general Hodge type case, similar objects can be defined, such as $p-Isog$, since Hodge type Shimura varieties can also be interepret as a ``moduli space" of abelian varieties with additional structures. My main purpose in this paper is to confirm the conjecture for GSpin Shimura varieties, generalizing the results by B\"ultel on GSpin$(5,2)$ to all $(n,2)$. 

The ordinary congruence relation is proved by Moonen in the PEL type case and generalized to Hodge type by Hong\cite{Hong} recently. In the cases where GSpin$(n,2)$ splits over $p$ (when $n$ is odd this is always the case, as we assume the group is nonramified over $p$), I show that the ordinary congruence relation is enough, that is, it implies the full version of the congruence relation. In the quasi-split but non split case, which only happens in the case where $n$ is even, I compute the Hecke polynomial and show that there is a factor killing the basic cycles. It is similar to the case dealt with by Koskivirta. 

Let GSpin$(V)$ be the general spinor group associated with a $\mathbb{Q}$-quadratic space $V$. Let $p$ be a prime number greater than $2$ such that GSpin$(V)$ is nonramified over $p$. I prove the following
\begin{theorem}
Let $\mu$ be the Hodge cocharacter of a GSpin Shimura variety. Let $H(X)$ be the Hecke polynomial of $\mu$, considered as a polynomial with coefficients in $\mathbb{Q}[p-Isog_k]$ via $\sigma\circ h$. Let $F$ be the Frobenius cycle as defined in \ref{twosections}. Then $H(X) = 0$.
\end{theorem}

The structure of the paper is the following: In section\ref{HkpB} I review the definition of the Hecke polynomial and B\"ultel's work \cite{B} on the ``group theoretic congruence relation". In section\ref{Stconj} I introduce the background needed to state conjecture \ref{conj} precisely for Hodge type Shimura varieties. In section\ref{ordcon} I review the ordinary congruence relation by Hong\cite{Hong}. In section\ref{shim} I review the definition of GSpin Shimura varieties. In section\ref{split} I prove the conjecture in the split cases; Finally in section\ref{quasi}, I compute the Hecke polynomial, showing that one can factor it in the same way as Koskivirta\cite{K} did so that a particular simple factor kills all possible basic cycles, and deduce the conjecture in the quasi-split non split case. 

\section{Hecke polynomial and B\"ultel's group theoretic congruence relation}\label{HkpB}
In this section I review the definition of the Hecke polynomial and a result proved by B\"ultel in his thesis. The references are Blasius-Rogawski\cite{BR} and Wedhorn\cite{W}. 

\vspace{.1in}
\subsection{The Hecke algebras} Fix an odd prime $p$. Let $F$ be a finite extension of $\mathbb{Q}_p$, $\varpi$ its uniformizer and $O_F$ be its ring of integers. Fix an algebraic closure $\overline{F}$ of $F$ and let $Gal(\overline{F}/F)$ be its Galois group. Let $F^{nr}$ be the maximal nonramified extension of $F$ inside $\overline{F}$. Let $\sigma$ be the Frobenius of $F$ as an element of $Gal(F^{nr}/F)$. Also, let $q$ be the cardinality of the residue field of $F$. 

Let $G$ be a nonramified reductive group over $F$, i.e. quasi-split over $F$ and splits over a nonramified extension of $F$. Let $K$ be a hyperspecial subgroup of $G(F)$. According to Bruhat-Tits\cite{BT}, take an $S$ which is a maximal split torus of $G$ defined over $F$, then $T = C_G(S)$ is a maximal torus of $G$ defined over $F$. Let $B$ be a Borel subgroup containing $T$ defined over $F$ and $U$ its unipotent radical. Let $W=N_G(T)/T$ be the Weyl group of $(G,T)$. It is an \'etale group scheme over $F$ and $W(\overline{F})$ is the Weyl group in the naive sense. The Borel $B$ determines a set of positive roots and coroots of $(G_{\overline{F}}, T_{\overline{F}})$. Let $\rho$ be the halfsum of all the positive roots. Let $K_T=K\cap T(F)$. 

I also need another parabolic pair $(M,P)$ such that $P$ is a parabolic subgroup defined over $F$ containing $B$, $M$ its Levi subgroup. Denote its unipotent radical by $N$ so that $P=M\ltimes N$. Let $K_P=K\cap P(F)$ and $K_N=K\cap N(F)$. Let $K_M$ be the image of $K$ under the projection:$P\rightarrow M$. The group $B\cap M$ is a Borel subgroup of $M$ and we use $U_M$ to denote the unipotent radical of this Borel of $M$. Use $dg$ to denote the left invariant Haar measure on $G$ normalized so that $K$ has volume one. Similarly define $dp, dn, dm$.

Define the spherical Hecke algebra with coefficients in a ring $A$
\begin{equation}
\mathscr{H}(G(F)//K,A), \mathscr{H}(M(F)//K_M,A),\mathscr{H}(T(F)//K_T,A)
\end{equation}
for $G, M, T$ respectively. For example, $\mathscr{H}(G(F)//K,A)$ as a set is defined to be all the finitely supported $A$-valued functions on the double cosets $K\backslash G(F)/K$, the group structure is just addition of functions and the ring structure is the convolution product. 

\subsection{The (twisted) Satake transform.} The twisted Satake transform for $G$ is defined to be
\begin{align}
\dot{\mathcal{S}}^G_T: \mathscr{H}(G(F)//K,A)&\longrightarrow \mathscr{H}(T(F)//K_T,A)\\
f &\longmapsto (t \mapsto \int_{U(F)}f(tu)du)
\end{align}
Similarly one can define this map for $M$
\begin{align}
\dot{\mathcal{S}}^M_T: \mathscr{H}(M(F)//K_M,A)&\longrightarrow \mathscr{H}(T(F)//K_T,A)\\
f &\longmapsto (t \mapsto \int_{U_M(F)}f(tu)du)
\end{align}
and
\begin{align}
\dot{\mathcal{S}}^G_M: \mathscr{H}(G(F)//K,A)&\longrightarrow \mathscr{H}(M(F)//K_M,A)\\
f &\longmapsto (m\mapsto \int_{N(F)}f(mn)dn)
\end{align}
We have the relation: $\dot{\mathcal{S}}^G_T=\dot{\mathcal{S}}_T^M\circ\dot{\mathcal{S}}^G_M$.

The Hecke algebra $\mathscr{H}(G(F)//K, A)$ is a polynomial algebra. Since $G$ is quasi-split over $F$, from Cartan decomposition, there is the following identification
\begin{align}
X_*(S)&\longrightarrow T(F)/K_T\\
h:v&\longmapsto 1_{K_Tv(\varpi)K_T}
\end{align}
in which $\varpi$ is the uniformizer of $F$ as defined above. 

Since one is using the twisted Satake transform rather than the usual Satake transform, Wedhorn\cite{W} defined a ``dot action" of $W(F)$ on $A[X_*(S)]$
\begin{equation}
w\cdot h_v = q^{\langle\rho, v-w(v)\rangle}h_{w(v)}
\end{equation}
where $h_v$ is short for $h(v)$. Under the twisted Satake transform, the Hecke algebra $\mathscr{H}(G(F)//K,A)$ is identified with the $W_M(F)$ invariant subalgebra. Therefore $A[X_*(S)]^{W_M(F)}$ can be viewed as a ring extension of $A[X_*(S)]^{W(F)}$ inside $A[X_*(S)]$. In\cite{B}, B\"ultel defined a distinguished element of $A[X_*(S)]^{W_M(F)}$
\begin{equation}\label{gFr}
g_{[\mu]}=1_{K_M\mu(\varpi)K_M}
\end{equation}
which will play the role of the Frobenius in the ordinary congruence relation. In this work, I will only consider the $\mathbb{Q}$-coefficients spherical Hecke algebra, and will use $\mathscr{H}(G(F)//K)$ for $\mathscr{H}(G(F)//K,\mathbb{Q})$. 

\subsection{The Hecke polynomials and the group theoretic congruence relation.} Let $\hat{G}$ be the dual group of $G_{\overline{F}}$. Fix a pinning of the root datum of $(G,B,T)$ in which $(G,B,T)$ are those defined in the previous section. When $G$ is quasi-split, define $^{L}G$ to be $Gal(F^{nr}/F)\ltimes\hat{G}$ since this paper only cares about the nonramified Langlands parameters. 

Let $l$ be a prime number different from $p$, the residue characteristic of $F$. Let $\overline{\mathbb{Q}}_l$ be the algebraic closure of the field $\mathbb{Q}_l$. Define the set of nonramified Langlands parameters $\Phi_{nr}(G)$ to be the set of $\hat{G}(\overline{\mathbb{Q}}_l)$ conjugacy classes of the homomorphisms $\Gamma_F=Gal(F^{nr}/F)\longrightarrow{^{L}G}(\overline{\mathbb{Q}}_l)$ such that $\phi(\sigma)=(\sigma,g_{\phi})$ in which $g_{\phi}$ is a semisimple element of $\hat{G}(\overline{\mathbb{Q}}_l)$. The set $\Phi_{nr}(G)$ can be indentified with the $\sigma$-conjugacy classes of semisimple elements of $\hat{G}(\overline{\mathbb{Q}}_l)$.

Hecke polynomial is defined for any cocharacter of $G_{\overline{F}}$, not only for miniscule ones. Consider a tower of field extension $F\subset E\subset F^{nr}$ and let $n = [E:F]$. Also let $\Gamma_E=Gal(F^{nr}/E)$. Let $\mu_E$ be an $E$-rational cocharacter of $G_E$, this in turn gives a highest weight module of $\hat{G}$, denote it by $V^{\mu_E}$. Similarly to above, define ${^{L}G}_E$ to be $\Gamma_E\ltimes\hat{G}$ where the action of $\Gamma_E$ on $\hat{G}$ is the restriction of $\Gamma_F$. Then 
\begin{proposition}
$r_{\mu_E}$ can be extended to a representation of ${^{L}G}_E$. In other words, there is a representation
\begin{equation}
r_{\mu_E}:{^{L}G}_E\longrightarrow GL(V^{\mu_E})
\end{equation}
such that it gets back the highest weight module of $\hat{G}$ with highest weight $\mu_E$ when restricted $r_{\mu_E}$ to $1\ltimes\hat{G}$. 
\end{proposition}
This is stated by Wedhorn\cite{W}, 2.5. Actually this is the corollary of the following well-known proposition
\begin{proposition}
Let $G$ be a reductive group over an algebraically closed field $K$ and $r:G\longrightarrow GL(V)$ be an irreducible algebraic representation. Let $\Lambda$ be a finite cyclic group of order $n$. Suppose we have a semidirect product $\Gamma\ltimes G$ such that for any $\gamma\in\Gamma$, we have $r\circ\gamma\simeq r$ as a representation of $G$. Then we can extend the representation $r$ to $\Gamma\ltimes G$
\end{proposition}
\begin{proof}
First, for any $\gamma\in\Gamma$, from $r\circ\gamma\simeq r$, we know that there exists an endomorphism of $V,f_\gamma:V\longrightarrow V$ such that $f_{\gamma}\circ r(g) = r(g)$ for any $g\in G$. However, for any two elements $\gamma_1,\gamma_2$ of $\Gamma$, $f_{\gamma_1\circ\gamma_2}$ is different from $f_{\gamma_1}\circ f_{\gamma_2}$ by a non zero constant $c_{\gamma_1,\gamma_2}$,since irreducible representations over an algebraically closed field has automorphisms defined by scalars. Consider three elements of $\Gamma,\gamma_i$ in which $i = 1,2,3$. Then there is a relation of the constants: $c_{\gamma_1\circ\gamma_2,\gamma_3}\cdot c_{\gamma_1\circ\gamma_3}=c_{\gamma_1,\gamma_2\circ\gamma_3}\cdot c_{\gamma_2\circ\gamma_3}$. This is just a cocycle in the cohomology group $H^2(\Gamma,K^*)$, in which we view $K^*$ as a trivial $\Gamma$ module. but from basic knowledge of group cohomology for finite cyclic groups, we know that this cohomology is nothing but $K^*/(K^*)^n$, which is $1$ since $K$ is algebraically closed. The vanishing of this cohomology groups implies that we rescale $f_{\gamma}$s so that $f_{\gamma_1\circ\gamma_2} = f_{\gamma_1}\circ f_{\gamma_2}$ holds, therefore extending the representation to all of $\Gamma\ltimes G$. 
\end{proof}

Now for any $\overline{\mathbb{Q}}_l$-algebra $R$, and $g\in\hat{G}(R)$, consider the Hecke polynomial
\begin{equation}
H_{G,\mu_E}(X) = det(X\cdot Id - q^{nd}r_{\mu_E}((\sigma\ltimes g)^n))
\end{equation}
in which $d=\langle\mu_E,\rho\rangle$. One can view this polynomial as a polynomial with coefficients in $\mathcal{O}(\hat{G})$. By this definition, $\mathcal{O}(\hat{G})$ is an algebra over $\overline{\mathbb{Q}}_l$. By Wedhorn\cite{W} proposition (2.7), it is actually a ploynomial with coefficients in $\mathcal{O}(\hat{G})$ where $\hat{G}$ is defined over $\mathbb{Q}$. 

Let $\mathcal{O}(\hat{G})^\sigma$ be the elements of $\mathcal{O}(\hat{G})$ which are invariant under the $\sigma$-twisted conjugation. More precisely, the twisted conjugation
\begin{align*}
\hat{G}\times\hat{G}&\xrightarrow{Ad\sigma}\hat{G}\\
(g,x)&\longmapsto gx
\end{align*}
induces a map on the rings
\begin{align*}
\mathcal{O}(\hat{G})&\xrightarrow{Ad\sigma^\lor}\mathcal{O}(\hat{G})\otimes\mathcal{O}(\hat{G})
\end{align*}
The invariant elements are those $f\in\mathcal{O}(\hat{G})$ such that $Ad\sigma^\lor(f)=1\otimes f$.
Similar let $\mathcal{O}(\hat{T})^\sigma$ be the $T$ counterpart. Then
\begin{proposition}
$H_{G,\mu_E}(X)$ has coefficients in $\mathcal{O}(\hat{G})^\sigma$, and there are isomorphisms: $\mathcal{O}(\hat{G})^\sigma\simeq \mathcal{O}(\hat{T})^\sigma\simeq \mathcal{O}(\hat{S})^{W(F)}$ in which the isomorphisms are given by restrictions. Therefore, we can view the Hecke polynomial as a polynomial with coefficients in $\mathcal{O}(\hat{S})^{W(F)}$. 
\end{proposition}

Now it is the time to state the ``group theoretic congruence relation" proved by B\"ultel (see B\"ultel\cite{B}, proposition (3.4), or Wedhorn\cite{W} (2.9))
\begin{proposition}
As an element of $\mathbb{Q}[X_*(S)]=\mathcal{O}(\hat{S})$, the $g[\mu]$ defined in \ref{gFr}, is a root of the Hecke polynomial, viewed as an element of $\mathcal{O}(\hat{S})^{W(F)}[X]$. 
\end{proposition}

\section{Statement of the congruence relation conjecture}\label{Stconj}
In this section I state the Chai-Faltings' version of the congruence relation \ref{conj} precisely for Hodge type Shimura varieties. 

\subsection{Moduli space of p-isogenies}\label{mspi} In his paper\cite{Ki}, Kisin proved the extension property and the smoothness of the semi-global integral model of Shimura varieties with the hyperspecial level structure. He also showed that it carries a ``universal family" of abelian varieties with the Hodge tensors determined by a Hodge embedding of its Shimura data.

Let $(G,h)$ be a Hodge type Shimura datum, in which $G$ is a reductive group over $\mathbb{Q}$ and $h:\mathbb{S}\longrightarrow G_{\mathbb{R}}$ a $G$-Hodge structure. Let $\mu$ be the Hodge cocharacter associated with $h$. Fix a prime $p$ over which $G$ is nonramified, i.e., $G_{\mathbb{Q}_p}$ is quasi-split and splits over $\mathbb{Q}_p^{nr}$. In this case there exists a reductive integral model of $G$ over $\mathbb{Z}_{(p)}$, so it is possible to define $K_p=G(\mathbb{Z}_p)$. Also choose a compact open subgroup $K^p\subset G(\mathbb{A}_f^p)$ which is small enough so that Sh$_{G,K_pK^p}$ is a variety over the reflex field of $(G,\mu)$, say $E$. Now let $(\text{GSp}_{2g},h_{2g})$ be a Siegel Shimura datum associated to a symplectic space over $\mathbb{Q}$ which admits a self-dual lattice $\Lambda_p$ in $V_{\mathbb{Q}_p}$. Let $r:G\longrightarrow\text{GSp}_{2g}$ be a symplectic faithful representation such that $r\circ h=h_{2g}$. Kisin showed that $G$ can be identified with the pointwise fixer of a set of tensors $(s_0)\subset V^\otimes$ inside of GSp$_{2g}$. By Zahrin's trick, one can assume that $r$ is defined over $\mathbb{Z}_{(p)}$. Let $K'_p=\text{GSp}_{2g}(\mathbb{Z}_p)$ and ${K'}^p$ small enough but still containing $r(K^p)$, one obtains the following morphism of varieties
\begin{equation}
Sh_{G,K_pK^p}\longrightarrow Sh_{GSp_{2g},K'_p{K'}^p,E}
\end{equation}
The Siegel modular variety has a smooth semi-global integral model over the ring $\mathbb{Z}_{(p)}$ by the construction of Mumford, let's call it $\mathscr{S}$ . Let $v$ be a prime of $E$ over $p$, and let $O_{E,(v)}$ be the localization of $O_E$ at $v$, and $k$ be the residue field $O_{E,(v)}/(v)$. Let $\mathscr{S}_{2g,K'_p{K'}^p,O_{E,(v)}}=\mathscr{S}_{2g,K'_p{K'}^p}\times_{\mathbb{Z}_{(p)}} O_{E,(v)}$. Then following Milne, the integral model for Sh$_{G,K_pK^p}$ is defined as the normalization of the closure of the image of the following morphism
\begin{equation}
Sh_{G,K_pK^p}\longrightarrow Sh_{GSp_{2g},K'_p{K'}^p,E}\longrightarrow\mathscr{S}_{2g,K'_p{K'}^p,O_{E,(v)}}
\end{equation}
Denote $Sh_{G,K_pK^p}$ by $Sh$ and $\mathscr{S}$ for short since I will always  fix a small enough $K^p$. Also, let $\mathscr{S}_k$ be special fiber for $\mathscr{S}$. Kisin proved that $\mathscr{S}$ is a smooth scheme over $O_{E,(v)}$ and it acquires a family of abelian schemes $A$ with a family of Hodge tensors $(s_{et,\alpha},s_{dR,\alpha},s_{cris,\alpha})$ by pulling back the universal family over $\mathscr{S}_{2g,K'_p{K'}^p,O_{E,(v)}}$. More precisely, there is a universal abelian scheme $A_0$, together with a prime to $p$ polarization $\lambda$ and a level structure
\begin{equation*}
\eta_0\in H^0(\mathscr{S}_{2g,K'_pK^{'p},O_{E,(v)}},Isom_{GSp}(V_{\mathbb{A}_f^p},(R^1f_*\mathbb{A}^p_f)^\lor)/{K'}^p)
\end{equation*}
where the subscript GSp means carrying the symplectic form on $V$ to the Weil form on $A_0$ for all prime to $p$ Tate modules. Then $(A,\lambda)$ is the pull back of $(A_0,\lambda_0)$. It also has a universal $K^p$-level structure derived from $\eta_0$, but I need to know the precise meaning of the tensors first
\begin{itemize}
\item the $l$-adic \'etale tensor: Let $S$ be a scheme over $E$. For any $S$ point, $S\longrightarrow Sh_{G,K_pK^p}$, pulling back the ``universal family" one gets an abelian scheme up to prime to $p$ quasi-isogeny:$A\longrightarrow S$. There is a local system of rank $2g: R_1f_*\mathbb{Q}_l$. A tensor is a morphism of local systems:$\mathbb{Q}_l\longrightarrow(R_1f_*\mathbb{Q}_l)^\otimes$.
\item the crystalline tensor: Let $S$ be a scheme over $k$. For any $S$ point:$S\longrightarrow\mathscr{S}_k$, there is an associated abelian variety $A\longrightarrow S$. It has a crystal: $\mathbb{D}(A)$ over the crystalline site $(S/W(k))$. A tensor is a morphism of crystals: $\mathbb{D}(\mathbb{Q}_p/\mathbb{Z}_p)\longrightarrow\mathbb{D}(A)^\otimes$. 
\item the de Rham tensor: Let $S$ be a scheme over $E$. For any $S$ point: $S\longrightarrow Sh_{G,K_pK^p}$, the associated abelian variety $A\longrightarrow S$. A tensor is a morphism $\mathcal{O}_S[1/p]\longrightarrow R^1f_*\mathcal{O}_A[1/p]$.
\end{itemize}
Then the tensors $(s_0)$ determine tensors for each type, which are denoted by $(s_{et,\alpha})$ for \'etale tensors collectively for all primes and $s_{cris,\alpha}$ for the crystalline tensors. The pull back of $\eta_0$ to $\mathscr{S}$ actually reduces to a section of  Isom$_G(V_{\mathbb{A}_f^p},(R^1f_*\mathbb{A}_f^p)^\lor)/K^p$, in which the subscript $G$ means $\eta$ should carry the standard tensors $(s_0)$ to the tensors $(s_{et,\alpha})$ of $A$ over $\mathscr{S}$. Therefore one obtains a ``universal family" over $\mathscr{S}$, namely $(A,\lambda,(s_{et,\alpha},s_{dR,\alpha},s_{cris,\alpha}),\eta)$. 

Now define the moduli space of $p$-isogenies (or rather $p$-quasi-isogenies). Let $p-Isog$ be a functor from $O_{E,(v)}$-schemes to groupoids, such that for a scheme $S$ over $O_{E,(v)}$, $p-Isog(S)$ consists of the following data
\begin{itemize}
\item Two $S$-points of $\mathscr{S}:x_i:S\longrightarrow\mathscr{S}$, for $i = 1,2$. It corresponds two abelian varieties over $S$ with additional structures: $(A_i,\lambda_i,(s_{et,\alpha,i},s_{dR,\alpha,i},s_{cris,\alpha,i}),\eta_i)$;
\item A $p$-quasi-isogeny $f$ up to prime to $p$ quasi-isogeny
\begin{equation}
(A_i,\lambda_1,(s_{et,\alpha,1},s_{dR,\alpha,1},s_{cris,\alpha,1}),\eta_1)\longrightarrow(A_2,\lambda_2,(s_{et,\alpha,2},s_{dR,\alpha,2},s_{cris,\alpha,2}),\eta_2)
\end{equation}
such that $f$ preserves the tensors, more precisely this means: For any geometric point $x:\text{Spec}\overline{K}\longrightarrow S$ and
\begin{enumerate}
\item If $\overline{K}$ is of characteristic $0$, $f$ carries the \'etale tensors $(s_{et,\alpha,1})$ to $(s_{et,\alpha,2})$ on the induced maps on the rational \'etale homologies of $A_i|_x$ for all primes (or cohomologies, where the map is the other way around);
\item If $\overline{K}$ is of characteristic $p$, then $f$ carries $s_{cris,\alpha,2}$ to $s_{cris,\alpha,1}$ on the rational crystalline cohomology of the abelian varieties, $\mathbb{D}(A_i|_x)(W(\overline{K}))[\frac{1}{p}]$.
\end{enumerate}
\item $f$ preserves the level structure, i.e. $f\circ\eta_1=\eta_2$ for the $f$ induced map on the prime to $p$ \'etale homologies (Tate modules). 
\end{itemize}

The pullback of the polarization $\lambda_2$ to $A_1$ equals $\gamma\cdot\lambda_1$ for some $\gamma$ well defined up to $p$-units (i.e. elements of $\mathbb{Z}^*_{(p)}$). Let $\gamma=\gamma'\cdot p^c$ for $\gamma'\in\mathbb{Z}^*_{(p)}$ and a unique $c\in\mathbb{Z}$. Let's call $p^c$ the multiplicator of $f$ and define $p-Isog^{(c)}$ to be the subfunctor of $p-Isog$ classifying quasi-isogenies with multiplicator $p^c$. Since the multiplicator is locally constant on $S$ for any point in $p-Isog(S)$, $p-Isog^{(c)}$ is an open and closed subfunctor of $p-Isog$. In other words, it consists of some connected components of $p-Isog$. 

By sending a $p$-isogeny tuple to $A_1$ and $A_2$, one gets the source projection $s$ and the target projection $t$
\begin{equation}
\xymatrix { & \ar[ld]_s p-Isog \ar[rd]^t & \\
\mathscr{S} & & \mathscr{S}
}
\end{equation}
As seen later, the elements of the Hecke algebra can define quasi-isogenies between abelian varieties. Since the coefficients of the Hecke polynomial all define genuine isogenies, from now on I only consider the components of $p-Isog$ for genuine isogenies, and will still use $p-Isog$ to mean the union of all these components. In this case $c\geq 0$.

The $p-Isog$ has the following property
\begin{proposition}
The functor $p-Isog$ is representable by a scheme whose connected components are quasi-projective. Also the two projections $s$ and $t$, when restricted to each $p-Isog^{(c)}$, are proper over $\mathscr{S}$. 
\end{proposition}
\begin{proof}
The proof follows from Hida\cite{Hida} and Chai-Faltings\cite{FC}. First consider the following diagram
\begin{equation*}
\xymatrix{ & \ar[ld]_s \mathscr{S}\times_{O_{E,(v)}}\mathscr{S} \ar[rd]^t & \\
\mathscr{S} & & \mathscr{S}}
\end{equation*}
Pullback the universal abelian scheme $A$ over $\mathscr{S}$ along $s$ and $t$, one gets: $s^*A$ and $t^*A$, together with additional structures. Consider the following functor over $\mathscr{S}\times_{O_{E,(v)}}\mathscr{S}$
\begin{equation*}
Hom_{\mathscr{S}\times_{O_{E,(v)}}\mathscr{S}}(s^*A,t^*A)
\end{equation*}
The $Hom$ means the morphism as abelian schemes (up to prime to $p$ quasi-isogeny), it doesn't have to be an isogeny, nor have to preserve other additional structures. By Hida\cite{Hida} Theorem 6.6, this functor is a scheme over $\mathscr{S}\times_{O_{E,(v)}}\mathscr{S}$ whose connected components are all quasi-projective. Then from the fact that being an isogeny and preserving level structures are locally constant, these two conditions cut out certain connected components of this $Hom$ scheme. For simplicity, I still use $Hom_{\mathscr{S}\times_{O_{E,(v)}}\mathscr{S}}(s^*A,t^*A)$ to mean the union of these components.

One also has to take care of the issue of preserving the tensors. First there is a diagram
\begin{equation*}
\xymatrix{p-Isog \ar[r] \ar[rd]_{(s,t)} & Hom_{\mathscr{S}\times_{O_{E,(v)}}\mathscr{S}}(s^*A, t^*A) \ar[d]^{(s,t)}\\
 &  \mathscr{S}\times_{O_{E,(v)}}\mathscr{S}
}
\end{equation*}
The horizontal arrow is an embedding.  I show that this is an open and closed embedding. This fact makes $p-Isog$ a union of connected components of the scheme $Hom_{\mathscr{S}\times_{O_{E,(v)}}\mathscr{S}}(s^*A,t^*A)$. To this purpose, I need to show the condition of preserving the tensors is locally constant.

Let $f:A_1\longrightarrow A_2$ be a point of $S$, for $S$ a connected component of $p-Isog$ not concentrated in characteristic $p$. Suppose $f|_x:A_1|_x\longrightarrow A_2|_x$ preserves the \'etale tensors at a geometric point Spec$\overline{K}\longrightarrow S$ of characteristic $0$. This means the section $f(s_{et,\alpha,1}) - s_{et,\alpha,2}$ of the \'etale local system $H_1(S,\mathbb{A}_f^p)$ is $0$ at this geometric point. Since a section of a local system vanishes on an entire connected component as long as it vanishes at one of its point. For a geometric point $y$ of characteristic $p$, preserving crystalline tensors means that $f|^*_ys_{cris,\alpha,2} - s_{cris,\alpha,1}=0$. One can apply the crystalline $p$-adic cohomology comparison theorem to show that $f$ preserves the crystalline tensor at this point, since $f_y$ is liftable to characteristic $0$. 

If there is a component concentrated in characteristic $p$, one has to show that if $f$ preserves the crystalline tensors at one geometric point $x$, then it also preserves the crystalline tensors at any other geometric point.  At $x$, say $f|_x^*s_{cris,\alpha,2} = s_{cris,\alpha,1}$. Applying Lemma 5.10 in \cite{MP} to $\mathbb{D}(A)^\otimes$ and the sections $f^*s_{cris,\alpha,2}-s_{cris,\alpha,1}=0$, concluding that it also vanishes at any other geometric point $y$.

For the properness of $s$ (or $t$): $p-Isog^{(c)}\longrightarrow\mathscr{S}$. One can apply the valuative criterion of properness used by Chai-Faltings\cite{FC}.
\end{proof}

\begin{remark}
The components concentrated in characteristic $p$ parametrize the isogenies that are not liftable to characteristic $0$. 
\end{remark}

From now on, use $\mathscr{S}_k^{ord}$ for the $\mu$-ordinary locus of $\mathscr{S}$ and $\mathscr{S}_{k,b}$ for its Newton stratum of Newton type $[b]$. Also let $\mathscr{S}_{\overline{k}},\mathscr{S}^{ord}_{\overline{k}}$ and $\mathscr{S}_{\overline{k},b}$ be the base change to the algebraic closure of $k$. For each irreducible component of $p-Isog_k$ define its Newton type
\begin{definition}
Let $Z$ be an irreducible component of $p-Isog_k$. $Z$ has Newton type $b$ if there is an open subset $U\subset Z$ such that $s(U)\subset\mathscr{S}_{k,b}$. 
\end{definition}

\begin{proposition}\label{Sstrat}
Let $Z$ be an irreducible component of $p-Isog_k$ of Newton type $b$. Then $s(U)\subset\cup_{b'\leq b}\mathscr{S}_{k,b'}$.
\end{proposition}

Now following Moonen\cite{M}, let's define $\mathbb{Q}[p-Isog_E]$,$\mathbb{Q}[p-Isog_k]$ and $\mathbb{Q}[p-Isog_k^{ord}]$. Let $O_{E,(v)}\longrightarrow L$ be a homomorphism of rings, in which $L$ is a field. Let $Z_{\mathbb{Q}}(p-Isog_L)$ be the vector space of algebraic cycles on $p-Isog_L$ with $\mathbb{Q}$ coefficients. Using the following homomorphism, this vector space is made into an algebra
\begin{equation}\label{mult}
p-Isog_L\times_{t,\mathscr{S}_L,s}p-Isog_L\longrightarrow p-Isog_L
\end{equation}
On $S$-points, it is given by the composition
\begin{equation*}
(A_1\xrightarrow{f_1}A_2),(A_2\xrightarrow{f_2}A_3)\longmapsto(A_1\xrightarrow{f_2\circ f_1}A_3)
\end{equation*}
In particular there are algebras $Z_{\mathbb{Q}}(p-Isog_{\mathbb{Q}})$ and $Z_{\mathbb{Q}}(p-Isog_k)$. Let $\mathbb{Q}[p-Isog_E]$ and $\mathbb{Q}[p-Isog_k]$ be the subalgebra of $Z_{\mathbb{Q}}(p-Isog_E)$ and $Z_{\mathbb{Q}}(p-Isog_k)$ generated by the irreducible components. Moonen proved
\begin{lemma}
The underlying vector space of $\mathbb{Q}[p-Isog_E]$ is the subspace of $Z_{\mathbb{Q}}(p-Isog_E)$ spanned by the irreducible components of $p-Isog_E$. 
\end{lemma}

Over the special fiber $\mathscr{S}_k$, it is more complicated since neither the source nor the target is finite. But over the $\mu$-ordinary locus the story is simpler. Define $p-Isog_k^{ord}$ to be the inverse image of $\mathscr{S}_k^{ord}$ under the source of the the target map. Following the same recipe above, there is the algebra $\mathbb{Q}[p-Isog^{ord}_k]$. Moonen proved that the above lemma is still true for $\mathbb{Q}[p-Isog_k^{ord}]$, i.e. the underlying subspace of this algebra is the same as the space spanned by the irreducible components of $p-Isog^{ord}_k$. 
There is a map
\begin{equation*}
ord:p-Isog_k\longrightarrow p-Isog^{ord}_k
\end{equation*}
defined by intersecting the ordinary locus. Another map
\begin{equation*}
cl:p-Isog^{ord}_k\longrightarrow p-Isog_k
\end{equation*}
defined by taking the closure of the ordinary locus. Extending by $\mathbb{Q}$-linearity, there are algebra homomorphisms
\begin{align}
ord:\mathbb{Q}[p-Isog_k]&\longrightarrow\mathbb{Q}[p-Isog^{ord}_k]\\
cl:\mathbb{Q}[p-Isog^{ord}_k]&\longrightarrow\mathbb{Q}[p-Isog_k]
\end{align}
Then $cl\circ ord$ just means ``dropping the irreducible components which are not $\mu$-ordinary". More precisely, there is an element in $\mathbb{Q}[p-Isog_k]$, say, $\Sigma c_i[Z_i]+\Sigma c_j'[Z_j]$ in which $Z_i$s are $\mu$-ordinary but $Z_j$s are not. Then the effect of $cl\circ ord$ just makes it $\Sigma c_i[Z_i]$. 

There is a specialization map from the generic fiber to the special fiber
\begin{equation*}
p-Isog_E\longrightarrow\mathbb{Z}[p-Isog_k]
\end{equation*}
defined by taking the closure of $p-Isog_E$ in $p-Isog$ then taking the reduction mod $p$. This in turn induces a homomorphism of vector spaces
\begin{equation*}
\sigma:\mathbb{Q}[p-Isog_E]\longrightarrow\mathbb{Q}[p-Isog_k]
\end{equation*}
Since the specialization preserves the dimension, the image of $\sigma$ is actually contained in the vector space spanned by the dimension $d=dimSh_{K_pK^p}$ cycles of $p-Isog_k$.

\subsection{Hecke algebra as algebra of p-isogenies}\label{pIsog} There is a homomorphism of algebras from the Hecke algebra $\mathscr{H}(G(\mathbb{Q}_p)//K_p)$ to the algebra $\mathbb{Q}[p-Isog_E]$. This is defined via the $\mathbb{Z}_p$-valued \'etale homology. Let $Z$ be an irreducible component of $p-Isog_E$, taking a geometric point Spec$\overline{K}\longrightarrow Z$. Then one gets a pair of abelian varieties over $\overline{K}$ and a $p$-isogeny $f$ between them. Taking their $p-adic$ Tate module, from the definition of $p-Isog$, one gets a linear map preserving Hodge tensors: $H_{1,et}(A_1,\mathbb{Z}_p)\longrightarrow H_{1,et}(A_2,\mathbb{Z}_p)$. Recall that there is a self-dual symplectic lattice $\Lambda_p$ defined in (3.1). There are isomorphisms $\gamma_i:H_{1,et}(A_i,\mathbb{Z}_p)\longrightarrow\Lambda_p$ carrying tensors $(s_{p,\alpha,i})$ to the standard tensors $(s_0)$. Therefore $\gamma_2\circ f\circ\gamma_1^{-1}$ induces an automorphism of $\Lambda_p\otimes\mathbb{Q}$ fixing the standard tensors. Hence it is an element of $G(\mathbb{Q}_p)$. Changing $\gamma_i$ amounts to changing $g$ to $h_1gh_2$ in which $h_1,h_2\in K_p=G(\mathbb{Z}_p)$. Therefore the coset $K_pgK_p$ is well defined. This coset is called the relative position of the $p$-isogeny. By locally constancy of relative positions, all the geometric points of $Z$ have the same relative position, so it is an invariant of irreducible components. Define the map
\begin{align}
K_p\backslash G(\mathbb{Q}_p)/K_p&\longrightarrow\mathbb{Q}[p-Isog_E]\\
K_pgK_p&\longmapsto\Sigma1\cdot{[Z]}
\end{align}
Where $Z$ runs through all the irreducible components of relative position $K_pgK_p$. Then the algebra homomorphism
\begin{equation}
h:\mathscr{H}(G(\mathbb{Q}_p//K_p)\longrightarrow\mathbb{Q}[p-Isog_E]
\end{equation}
is just the $\mathbb{Q}$-linear span of the above map. 

\begin{remark}
For the above map to make sense, $\sigma Z$ needs to be a finite sum. This is indeed the case. To see this, on the generic fiber, the source map $s$ is quasi-finite when restricted to the components of $p-Isog_E$ with a fixed relative position. Since $s$ is proper, $s$ is indeed finite. Therefore there are only finitely many such components. 
\end{remark}

Composing this homomorphism with the specialization map, there is a linear map
\begin{equation}\label{hecketoring}
\sigma\circ h:\mathscr{H}(G(\mathbb{Q}_p)//K_p)\longrightarrow\mathbb{Q}[p-Isog_k]
\end{equation}
Extending this map to polynomial rings with coefficients in each one of the algebras
\begin{equation}
\mathscr{H}(G(\mathbb{Q}_p//K_p)[X]\longrightarrow\mathbb{Q}[p-Isog_k][X]
\end{equation}
Therefore the Hecke polynomial in the previous section can be viewed as a polynomial with coefficients in the latter algebra, hence it is meaningful to talk about its roots in the latter algebra. 

\subsection{Two sections of the source map: $F$ and $\langle p\rangle$}\label{twosections} Next I define a section of the source map, i.e., a morphism of schemes\cite{M}:$\phi:\mathscr{S}_k\longrightarrow p-Isog_k$ such that $s\circ\phi=id$. Given an $S$ point $x$ on $\mathscr{S}_k$, it corresponds to an abelian variety over $S$ with additional structures: $(A,(s_{cris,\alpha}),\lambda,\eta)$. Then define
\begin{equation}
\phi:(A,(s_{cris,\alpha}),\lambda,\eta)\longmapsto(Fr: (A,(s_{cris,\alpha}),\lambda,\eta),\lambda,\eta)\longmapsto(A^{(p)},(s^{(p)}_{cris,\alpha}),\lambda^{(p)},\eta^{p}))
\end{equation}
in which $Fr$ is the relative Frobenius. To be a section of the source projection, $F$ must preserve the crystalline Hodge tensor and the level structure. This is indeed the case
\begin{proposition}
$Fr$ preserves the crystalline Hodge tensors, and $\eta^{(p)}$ satisfies the following: Let
\begin{equation}
V\otimes\mathbb{A}_f^p\xrightarrow{f}H_1(A,\mathbb{A}_f^p)\xrightarrow{Fr}H_1(A^{(p)},\mathbb{A}_f^p)
\end{equation}
be the level structure of $A$ and the induced map of $Fr$ on the prime to $p$ Tate module. Then $\eta^{(p)}=Fr\circ\eta$ mod $K^p$. Therefore
\begin{equation}
(F:(A,(s_{cris,\alpha}),\lambda,\eta)\longmapsto(A^{(p)},(s^{(p)}_{cris,\alpha}),\lambda^{(p)},\eta^{(p)}))
\end{equation}
is indeed a point of $p-Isog^{(1)}_k$. 
\end{proposition}
\begin{proof}
First let's review the definition of the crystalline tensors $(s^{(p)}_{cris,\alpha})$. Given a geometric point Spec$\overline{K}\longrightarrow\mathscr{S}_k$, there is the following Frobenius diagram
\begin{equation}
\begin{tikzcd} A \arrow[drr, bend left, "AbF"] \arrow[ddr] \arrow[dr, "Fr"] & & \\ & A^{(p)} \arrow[r, "\sigma"] \arrow[d] & A \arrow[d] \\ & \overline{K} \arrow[r, "\sigma"] & \overline{K}
\end{tikzcd}
\end{equation}
in which $AbF$ is the absolute Frobenius and $\sigma$ the arithmetic Frobenius. The tensors $(s^{(p)}_{cris,\alpha})$ is defined to be the pull back of $(s_{cris,\alpha})$ along the arithmetic Frobenius. Since the absolute Frobenius induces identity on the cohomology, the pull back of $(s^{(p)}_{cris,\alpha})$ along $Fr$ must coincide with $(s_{cris,\alpha})$. The same argument can prove the compatibility of the level structures.
\end{proof}
Define the image of this section again as $F$, the Frobenius cycle. For later use, I define another section of $s$, even though it is not used in the statement of \ref{conj}. 
Using Koskivirta's notation\cite{K}, this cycle is called $\langle p\rangle$. It is defined as the image of the section
\begin{equation}
\mathscr{S}_k\longrightarrow p-Isog:(A,(s_{cris,\alpha}),\lambda,\eta),\lambda,\eta)\xrightarrow{\times p}(A,(s_{cris,\alpha}),\lambda,p\cdot\eta))
\end{equation}
Koskivirta calls this section ``multiplication by $p$". Its image $\langle p\rangle$ is also the specialization of the cycles indexed by $K_ppK_p$ in the generic fiber. 
\begin{proposition}
When $K^p$ is small enough, the image of this $\langle p\rangle$ is indeed the specialization image of the components of $p-Isog_E$ indexed by the coset $pK_p$.
\end{proposition}
\begin{proof}
To see this, first recall what are those cycles indexed by $pK_p$: It is formed by the cycles with (\'etale) relative position $p\cdot Id$. Suppose there is a geometric point on this cycle, it corresponds to a pair $(A_1,(s_{et,\alpha,1}),\lambda,\eta)\longmapsto(A_2,(s_{et,\alpha,2}),\lambda,\eta)$ such that $f$ induce relative position $K_ppK_p$ on $\Lambda_p\otimes\mathbb{Q}$. Then consider $p^{-1}\cdot f$, it induces an isomorphism from $H_{1,et}(A_1,\mathbb{Z}_p)$ to $H_{1,et}(A_2,\mathbb{Z}_p)$, so it must be a prime to $p$ quasi-isogeny. So $f$ is equivalent to multiplication by $p$. The multiplication by $p$ cycle specializes to multiplication by $p$. 
\end{proof}

Now the conjecture is the following
\begin{conjecture}
Let $F$ be the Frobenius cycle defined as above, then $H(F)=0$ in $\mathbb{Q}[p-Isog_k]$. 
\end{conjecture}

\section{The ordinary congruence relation}\label{ordcon}
In section\ref{pIsog} I also introduced a ring $\mathbb{Q}[p-Isog_k^{ord}]$ and maps $ord$ and $cl$. Composing $\mathscr{H}(G(\mathbb{Q}_p)//K_p)\longrightarrow\mathbb{Q}[p-Isog_k]$ with $cl\circ ord$, there is a map
\begin{equation*}
\mathscr{H}(G(\mathbb{Q}_p//K_p)\longrightarrow\mathbb{Q}[p-Isog_k]\xrightarrow{cl\circ ord}\mathbb{Q}[p-Isog_k]
\end{equation*}
Then the ordinary congruence relation reads
\begin{theorem}
View $H(X)$ as a polynomial with coefficients in $\mathbb{Q}[p-Isog_k]$ by the map above. Then $cl\circ ord(F)$ is a root of this polynomial. 
\end{theorem}
Note that this theorem is weaker than the conjecture in the last section, because to prove that conjecture one also has to check the polynomial formed by the terms with non generically ordinary coefficients vanishes on $F$.

This theorem is known for all Hodge type Shimura varieties, since $\mathbb{Q}[p-Isog^{ord}_k]$ can be ``parametrized" by the spherical Hecke algebra $\mathscr{H}(M(\mathbb{Q}_p)//K_{M,p})$, in which $M$ is the centrailizer of the Hodge cocharacter of the Shimura variety. More precisely, there is a map:$\mathscr{H}(M(\mathbb{Q}_p)//K_{M,p})\longrightarrow\mathbb{Q}[p-Isog_k^{ord}]$ where $K_{M,p}$. There is a diagram
\begin{equation*}
\xymatrix{
\mathscr{H}(G(\mathbb{Q}_p)//K_p,\mathbb{Q}) \ar[r]_h \ar[d]_{\mathcal{S}^G_M} & \mathbb{Q}[p-Isog_E] \ar[d]_{ord\circ\sigma}\\
\mathscr{H}(M(\mathbb{Q}_p)//K_{M,p}),\mathbb{Q}) \ar[r]_{\overline{h}} & \mathbb{Q}[p-Isog_k^{ord}]
}
\end{equation*} 
The element $g_{[\mu]}$ defined by B\"ultel goes to $F$ in $\mathbb{Q}[p-Isog_k^{ord}]$. Therefore by this diagram the group theoretic congruence by B\"ultel to the ordinary congruence relation. 

Moonen\cite{M} used Serre-Tate theory to prove the existence of the above diagram in the PEL case, and Hong\cite{Hong} generalized Moonen's proof to the Hodge type case recently.

\section{GSpin Shimura varieties and its isocrystals}\label{shim}
In this section I review the GSpin Shimura varieties, the main objects of interests in this paper. A thorough theory of these is developed by Madapusi-Pera\cite{MP}. Zhang's thesis\cite{Zh} also has a good introduction to it. I also review the classification of its isocrystals of its good reduction. 

\subsection{GSpin Shimura datum}\label{gsd}
Since I need to use the integral model of the Shimura varieties later, let's start the definition over $\mathbb{Z}$. Let $(V,q)$ be a quadratic free module over $\mathbb{Z}$ of rank $N$ such that $V\otimes\mathbb{R}$ has signature $(N - 2,2)$; Let $p$ be a prime number such that $V\otimes\mathbb{Z}_p$ is self dual. Let SO$(V)$ be the special orthogonal group over $\mathbb{Z}$ defined by $V$. Let $C(V)=C^+(V)\oplus C^-(V)$ be the Clifford algebra attached to $V$. $V$ is naturally sitting inside $C^-(V)$ via left multiplication. Define the group scheme GSpin$(V)$ whose functor of points on an algebra $R$ is given by ${g\in C^+(V_R)^\times:gV_Rg^{-1}\subset V_R}$. Take an element $\delta\in C(V)^\times$ such that $\delta^*=-\delta$, then one can define a symplectic form $\psi(c_1,c_2)=Trd(c_1\delta c_2)$ on $C(V)$. Since $q$ is perfect on $V_{\mathbb{Z}_p}$, this symplectic form is perfect on $C(V_{\mathbb{Z}_p})$. The left multiplication of GSpin$(V)$ on $C(V)$ defines an embedding
\begin{equation}\label{embd}
\text{GSpin}(V)\longrightarrow\text{GSp}(C(V),\psi)
\end{equation}
As Kisin proved, GSpin is the pointwise fixer of a set of tensors $(s_0)\in C(V)^\otimes$. Madapusi-Pera gave a complete list of these tensors in\cite{MP}, 1.3. For later use, let $D$ be the linear dual of $C(V)_{\mathbb{Z}_p}$. Then $D^\otimes=C(V)^\otimes_{\mathbb{Z}_p}$. 

In this paper it is more convenient to work with the Shimura datum group theoretically. Choose a basis of $V_\mathbb{R}$ such that the quadratic form under this basis is given by
\[
\begin{bmatrix}
-1 & & & &\\
 &-1 & & &\\
 & & 1 & &\\
 & & & \ddots &\\
 & & & & 1
\end{bmatrix}
\]
Define $h:\mathbb{S}\longrightarrow\text{SO}(V_\mathbb{R})$ to be
\[
\begin{bmatrix}
\frac{a^2-b^2}{a^2+b^2} &\frac{2ab}{a^2+b^2} & & &\\
\frac{-2ab}{a^2+b^2} & \frac{a^2-b^2}{a^2+b^2} & & &\\
& & 1 & &\\
& & & \ddots & \\
& & & &1
\end{bmatrix}
\]
This cocharacter defines a Hodge structure on $V$, which is of weight $0$. Then by the Kuga-Satake construction there is a lifting of $h$ to GSpin$(V_\mathbb{R})$, giving rise to a Hodge structure of weight $1$ on the Clifford algebra
\begin{equation}
\xymatrix{ & \text{GSpin}(V_\mathbb{R}) \ar[d] \\
\mathbb{S} \ar[ur] \ar[r] & \text{SO}(V_{\mathbb{R}})}
\end{equation}
Let $\mu$ be the $\mathbb{C}$-cocharacter defined by $\mu(z)=h(z,1)$. Denote the lifting of $\mu$ corresponding to the Kuga-Satake lifting of $\mu$ by $\mu^{KS}$. The conjugacy class of this cocharacter actually descents to $\mathbb{Q}$, so the field of definition of the Shimura varieties is just $\mathbb{Q}$. For details of these facts, see Madapusi-Pera\cite{MP} section 3. 

\subsection{Isocrystals, Rapoport-Zink spaces} At $p$, since $q$ is self-dual on $V_{\mathbb{Z}_p}$, it is possible choose a basis ${e_1,...,e_n}$ of the quadratic free module so that the quadratic form is given by
\[
\begin{bmatrix}
0 & 1& & &\\
 1 & 0& & &\\
 & & \ast & &\\
 & & & \ddots &\\
 & & & & \ast
\end{bmatrix}
\]
Then $\mu$ is conjugate to the following $\mathbb{Z}_p$-cocharacter, which is still denoted by $\mu:\mathbb{G}_{\mathbb{Z}_p}\longrightarrow\text{SO}(V_{\mathbb{Z}_p})$
\[
\begin{bmatrix}
t & & & &\\
 & t^{-1} & & &\\
 & & 1 & &\\
 & & & \ddots &\\
 & & & & 1
\end{bmatrix}
\]
Its corresponding Kuga-Satake lift is
\begin{equation}
\mu^{KS}:\mathbb{G}_m\longrightarrow\text{GSpin}(V),t\longmapsto t^{-1}e_1e_2 + e_2e_1
\end{equation}

Now let $k=\mathbb{F}_p$ and $\overline{k}=\overline{\mathbb{F}}_p$. Let $W=W(\overline{k})$, the ring of Witt vectors of $k$ and $K=W(\overline{k})[1/p]$. Let $B(\text{GSpin},\mu^{KS})$ be the $\mu^{KS}$ admissible Kottwitz set for GSpin$(V_K)$ classifying $\sigma$-conjugacy classes $[b]$ of GSpin$(V)(K)$ whose Newton cocharacter $\nu_b$ is less than or equal to $\mu^{KS}$. The embedding of the group schemes
\begin{equation}
\text{GSpin}(V_{\mathbb{Q}_p})\longrightarrow\text{GL}(C(V_{\mathbb{Q}_p}))
\end{equation}
induced from \ref{embd} together with $\mu^{KS}$ and an element of $[b]\in B(\text{GSpin}(V_{\mathbb{Q}_p}),\mu^{KS})$ define a local nonramified Shimura-Hodge datum in the sense of \cite{HP}, definition 2.2.4. One can define the associated Rapoport-Zink space for the local nonramified Shimura-Hodge datum $(\text{GSpin},b,\mu,C)$. 

First, there exists a $p$-divisible group attached to $(\text{GSpin},b,\mu,C)$\cite{HP}, 2.2.6, denoted by $X_0=X_0(\text{GSpin},b,\mu,C)$ with a set of Frobenius invariant tensors $(s_{\alpha,0})$. Following their notations, let ANilp$_W$ be the category of $W$-algebras on which $p$ is nilpotent. Then the Rapoport-Zink space is defined as follows\cite{HP}, definition 2.3.3
\begin{definition}
Consider the set valued functor
\begin{equation*}
RZ=RZ_{(\text{GSpin},b,\mu,C)}:\text{ANilp}_W\longrightarrow\text{(Ens)}
\end{equation*}
whose functor of points on an $R\in\text{ANilp}_W$ is given by isomorphism classes of the following data: A triple $(X,(s_\alpha),\rho)$ consists of a $p$-divisible group over $R$ and a quasi-isogeny
\begin{equation*}
\rho:X_0\otimes_k\overline{R}\longrightarrow X\otimes_R\overline{R}
\end{equation*}
where $\overline{R}=R/pR$. And $(s_\alpha)$ is a set of crystalline tensor, i.e. morphisms of crystals
\begin{equation*}
s_\alpha:\mathbb{D}(\mathbb{Q}_p/\mathbb{Z}_p)\longrightarrow\mathbb{D}(X)^\otimes
\end{equation*}
over $R$ such that
\begin{equation*}
s_\alpha:\mathbb{D}(\mathbb{Q}_p/\mathbb{Z}_p)[1/p]\longrightarrow\mathbb{D}(X)^\otimes[1/p]
\end{equation*}
are Frobenius equivariant, satisfying the following properties
\begin{itemize}
\item For some nilpotent ideal $J\subset R$ containing $p$, the restriction of $s_\alpha$ to $R/J$ is identified with $s_{\alpha,0}$ under the isomorphism of the isocrystals induced by $\rho$
\begin{equation*}
\mathbb{D}(\rho):\mathbb{D}(X_{R/J})^\otimes[1/p]\xrightarrow{\sim}\mathbb{D}(X_0\times_k R/J)^\otimes[1/p]
\end{equation*}
\item The sheaf of $G_W$-sets over CRIS$(\text{Spec}(R)/W)$ given by the isomorphisms
\begin{equation*}
\underline{Isom}_{s_\alpha,s_0\otimes 1}(\mathbb{D}(X), D\otimes_{\mathbb{Z}_p}R)
\end{equation*}
is a crystal of $G_W$-torsors.
\item There exists an \'etale cover ${U_i}$ of Sprc$(R)$, such on each $U_i$ a tensor preserving isomorphism of vector bundles
\begin{equation*}
\mathbb{D}(X_{U_i})_{U_i}\xrightarrow{\sim}D\otimes_{\mathbb{Z}_p}\mathcal{O}_{U_i}
\end{equation*}
such that the Hodge filtration
\begin{equation*}
Fil^1(X_{U_i})\subset\mathbb{D}(X_{U_i})_{U_i}\xrightarrow{\sim}D\otimes_{\mathbb{Z}_p}\mathcal{O}_{U_i}
\end{equation*}
is induced by the cocharacter that is $G(U_i)$ conjugate to $\mu^{KS}$. 
\end{itemize}
Two such triples are equivalent if there exists a tensor preserving isomorphism between them commuting with the two quasi-isogenies.
\end{definition}

I refer to \cite{HP} for the definition of the big crystalline site CRIS$(\text{Spec}(R)/W)$. The above definition obviously works for any Hodge type Shimura varieties. One of their main results in \cite{HP} is that $RZ$ is representable by a locally formally finite type formal scheme.

Let $K_p=\text{GSpin}(V)(\mathbb{Z}_p)$ and $K_p'=\text{GSp}(C(V),\psi)(\mathbb{Z}_p)$. Also let $K^p\subset\text{GSpin}(V)(\mathbb{A}_f^p)$ and $K'^{p}=\text{GSp}(C(V),\psi)(\mathbb{A}_f^p)$ containing $K^p$. Then \ref{embd} induces an embedding of integral models
\begin{equation}
\mathscr{S}_{K_pK^p}\longrightarrow\mathscr{S}_{K'_pK'^{p}}
\end{equation}
This embedding of semi-global integral models determines a closed embedding\cite{HP}, proposition 3.2.11
\begin{equation}
RZ\longrightarrow RZ_{\text{GSp}_{C(V)}}
\end{equation}
This embedding provide the objects over $RZ$ one more structure, the polarization. Let $\lambda_0$ be the polarization defined by $\psi$ on the object $(X_0,(s_{\alpha,0}))$. Then the objects over in $RZ(R)$ for $R\in\text{ANilp}_W$ are 
\begin{equation*}
\rho:(X_0\otimes_k\overline{R},(s_\alpha),\lambda)\longrightarrow(X\otimes_R\overline{R},(s_{\alpha,0}),\lambda_0)
\end{equation*}
Then $\rho^\lor\circ\lambda\circ\rho=c^{-1}(\rho)\cdot\lambda_0$ for some $c(\rho)\in\mathbb{Q}_p^\times$. Based on $ord_p(c(\rho))$, $RZ$ is decomposed into open and closed sub formal schemes. Fix an integer $l$, let $RZ^{(l)}$ be the open and closed sub formal schemes on which $c(\rho)=p^l$. Call $RZ^{(l)}$ the component with multiplicator $p^l$. 

In this paper, I only consider the reduced locus of the Rapoport-Zink formal scheme, which are just schemes over $\overline{k}$, denoted by $RZ^{red}$ and $RZ^{red,(l)}$. For later use, I need the dimension of $RZ^{red}_b$ for all $[b]\in B(\text{GSpin}(V_{\mathbb{Q}_p}),\mu^{KS})$. The Newton cocharacter $\nu_b$s for all $[b]$ are needed. 

Since GSpin$(V)$ is the central extension of SO$(V)$, the following map
\begin{align*}
B(\text{GSpin}(V_{\mathbb{Q}_p}),\mu^{KS})&\longrightarrow B(\text{SO}(V_{\mathbb{Q}_p}),\mu)\\
b'&\longmapsto b
\end{align*}
induced by the surjection of GSpin onto SO is a bijection. 

Let's consider the $[b]$s for SO. Over $K$, using the anti-diagonal matrix as the metric matrix of the quadratic form. It is possible since SO splits over a nonramified extension of $\mathbb{Q}_p$
\[
\begin{bmatrix}
& & 1\\
& \ddots &\\
1 & &
\end{bmatrix}
\]
Then the Hodge cocharacter $\mu$ is
\[
\begin{bmatrix}
p & & & & \\
& 1 & & &\\
& & \ddots & &\\
& & & 1 &\\
& & & & p^{-1}
\end{bmatrix}
\]
the basic isocrystal is define by the class of the following $b$
\[
\begin{bmatrix}
 & & & & p\\
& 1 & & &\\
& & \ddots & &\\
& & & 1 &\\
p^{-1} & & & & 
\end{bmatrix}
\]
Since it squares to the identity matrix, its Newton cocharacter $\nu_b$
\begin{equation*}
(0,......,0)
\end{equation*}

Following the terminology of \cite{MP}, the Newton types between the $\mu$-ordinary and the basic are called finite height. The $b$s for the finite height isocrystal
\[
\begin{bmatrix}
& & & p & & & & & & & \\
1 & & & & & & & & & & \\
& \ddots & & & & & & & & &\\
& & 1 & & & & & & & &\\
& & & &1 & & & & & &\\
& & & & &\ddots & & & & &\\
& & & & & &1 & & & &\\
& & & & & & & & 1 & &\\
& & & & & & & & &\ddots &\\
& & & & & & & & & &1\\
& & & & & & & p^{-1} & & &
\end{bmatrix}
\]
where the size of the left upper and right lower matrices takes values between $2$ inclusive and $\lfloor dimV/2 \rfloor$. Let this size be $m$. Its Newton cocharacter is the following: $(\frac{1}{m},\dots,\frac{1}{m},0,\dots,0,-\frac{1}{m}\dots,-\frac{1}{m})$, since $b^m$ is the matrix
\[
\begin{bmatrix}
p & & & & & & & &\\
& \ddots & & & & & & &\\
& & p & & & & & &\\
& & &1 & & & & &\\
& & & & \ddots & & & &\\
& & & & & 1 & & &\\
& & & & & & p^{-1} & &\\
& & & & & & &\ddots &\\
& & & & & & & & p^{-1}
\end{bmatrix}
\]
With these data, the dimension of the Rapoport-Zink spaces can be computed using Rapoport's formula. Let $G$ be a reductive scheme over $\mathbb{Q}_p$, $K$ be the field as defined above, $J_b$ be the following $\mathbb{Q}_p$ group scheme
\begin{equation}\label{Kot}
J_b(R)=\{g\in G(K\otimes_{\mathbb{Q}_p}R): gb\sigma(g)^{-1}=b\}
\end{equation}
where $b\in B(G)$ and $R$ any $\mathbb{Q}_p$ algebra.
\begin{theorem}
(Zhu\cite{Zh}) The underlying reduced scheme of the Rapoport-Zink space associated with $(G,\mu,b)$ has dimension $\langle \mu-\nu_b,\rho\rangle-\frac{1}{2}def_G(b)$ in which $def_G(b)=rk_{\mathbb{Q}_p}G-rk_{\mathbb{Q}_p}J_b$.
\end{theorem}
Applying this formula to SO, $\mu$ and any finite height $b$. Divide it into two subcases according to wheather $N=dim(V)=2n$ or $2n+1$. Using the basis of $V_K$ such that the quadratic form is antidiagonal. Let $\chi_i$ be the following character of the diagonal maximal torus
\begin{equation*}
diag(t_1,t_2,\dots,t_n,t_n^{-1},\dots,t_2^{-1},t_1^{-1})\longmapsto t_i
\end{equation*}
in the even case;
\begin{equation*}
diag(t_1,t_2,\dots,t_n,1,t_n^{-1},\dots,t_2^{-1},t_1^{-1})\longmapsto t_i
\end{equation*}
in the odd case.

The SO$(V)$ has positive roots:$\{\chi_i\pm\chi_j\}$ for $i < j$ in the even case and $\{\chi_i\pm\chi_j\}$ for $i < j \cup \{\chi_i\}$ for all $i$ in the odd case. Since $\mu=(1,0,\dots,0,1)$ and $\nu_b=(1/m,1/m,\dots,-1/m,-1/m)$. Then $\langle\mu,\rho\rangle=n - 1$ if $N=2n$ and $\langle\mu,\rho\rangle=(2n-1)/2$ if $N = 2n + 1$.

Now let's compute $\langle\nu_b,\rho\rangle$. If $N=2n$, $2\langle\nu_b,\rho\rangle=(n - m)\cdot(2/m)\cdot m+(2/m)\cdot m(m - 1)/2=2n-m-1$; If $N=2n+1$,$2\langle\nu_b,\rho\rangle=(n - m)\cdot(2/m)\cdot m+(2/m)\cdot m(m - 1)/2 + m\cdot(1/m)=2n-m$. Therefore in each case, $2\langle\mu-\nu_b,\rho\rangle=m - 1$. 

The remaining work to do is to compute the $def_G$. Let's define the upper left corner matrix of the matrices for $b$s be $b_s$. It is nothing but a matrix who defines the basic GL$_m$ isocrystal, so $J_{b_s}=D^*_{1/m}$ the unit group of the division algebra of Hasse invariant $1/m$. This group is known to have $\mathbb{Q}_p$-rank 1 (It is isotropic modulo center, this $1$ comes from the center). In all possible cases, i.e.odd split, even split and even quasi-split but not split, one has
\begin{equation*}
J_b=J_{b_s}\times\text{SO}(N-2m)
\end{equation*}
in which SO$(N-2m)$ has the same splitting property as SO$(N)$ but smaller size. Therefore $def_G$ is always $m-1$. 

Plug into the Rapoport's formula
\begin{equation*}
\langle\mu-\nu_b,\rho\rangle-\frac{1}{2}def_G=\frac{1}{2}[(m-1)-(m-1)]=0
\end{equation*}
Since GSpin is the central extension of  SO, i.e.
\begin{equation*}
1\longrightarrow\mathbb{G}_m\longrightarrow\text{GSpin}\longrightarrow\text{SO}\longrightarrow 1
\end{equation*}
one can get another central extension, for $b'$ over $b$
\begin{equation*}
1\longrightarrow\mathbb{G}_m\longrightarrow J_{b'}\longrightarrow J_b\longrightarrow 1
\end{equation*}
Therefore the Rapoport's formula for $b'$ and $\mu^{KS}$ would give the same number as that for the corresponding $b$ and $\mu$. The reduced locus of the finite height GSpin Rapoport-Zink spaces have dimension $0$. 

\subsection{Rapoport-Zink uniformization of the basic locus.}
It is possible to apply the Rapoport's formula to compute the dimension of the basic Rapoport-Zink space. Here I directly cite the results from \cite{HP}, which also includes the uniformization map.
\begin{theorem}\label{unif}
(Howard-Pappas\cite{HP}) The dimension of the underlying reduced scheme of the basic Rapoport-Zink space of GSpin$(N-2,2)$ Shimura variety with hyperspecial level structure is the following
\begin{enumerate}
\item $(N-3)/2$, if $N$ is odd (in this case GSpin$(V_{\mathbb{Q}_p})$ is always split at $p$, since I already assume it is quasi-split and there is no Dykin diagram automorphism in the odd case);
\item $(N-4)/2$, if $N$ is even and GSpin$(V_{\mathbb{Q}_p})$ is split at $p$;
\item $(N-2)/2$, if $N$ is even and GSpin$(V_{\mathbb{Q}_p})$ is quasi-split and non-split at $p$. 
\end{enumerate}
And there is the Rapoport-Zink uniformization map
\begin{equation}
\Theta_b:I(\mathbb{Q})\backslash RZ_b\times G(\mathbb{A}_f^p)/K\longrightarrow(\hat{\mathscr{S}}_W)_{\mathscr{S}_{\overline{k},b}}
\end{equation}
\end{theorem}
in which $I$ is an inner form of GSpin$(V_{\mathbb{Q}})$ such that $I_{\mathbb{Q}_l}$ is isomorphic to GSpin$(V_{\mathbb{Q}_l})$ and $I_{\mathbb{Q}_p}$ is the automorphism group of the basic isocrystal $J_b$ defined by Kottwitz, i.e. \ref{Kot}, which is an inner form of GSpin$(V_{\mathbb{Q}_p})$. 

More explicitly, given a $\overline{k}$ point of the Rapoport-Zink space, $(\rho:X_0\longrightarrow X)$, by multiplying $\rho$ by a large enough $p$-power, say $p^a$, one obtains a genuine isogeny: $p^a\cdot\rho:X_0\longrightarrow X$. Let $A_{x_0}$ be an abelian variety whose $p$-divisible group is isomorphic to $X_0$ and fix such an isomorphism, and fix a level structure of $A_0$. Then the kernel of this map $ker(p^a\cdot\rho)$ is a finite subgroup scheme of $A_{x_0}$. Let $A$ be the abelian variety $A_{x_0}/ker(p^a\cdot\rho)$. Then one gets a genuine isogeny of abelian varieties
\begin{equation*}
p^a\cdot\rho:A_{x_0}\longrightarrow A
\end{equation*}
Dividing this by $p^a$, it becomes a quasi-isogeny $\rho:A_{x_0}\longrightarrow A$ whose $p$-divisible groups recover the quasi-isogeny of $p$-divisible groups:$\rho:X_0\longrightarrow X$. Taking various kinds of cohomology functors of this quasi-isogeny of abelian varieties, and transfer the tensors, polarization and the level structure to $A$. Let them be $\lambda, (s_{et,\alpha},s_{dR,\alpha},s_{cris,\alpha}),\eta$. Then there is a map
\begin{align}
RZ_b\times G(\mathbb{A}_f^p)&\longrightarrow\mathscr{S}_{\overline{k},b}\\
((\rho:X_0\longrightarrow X),g)&\longmapsto(A,\lambda,(s_{et,\alpha},s_{dR,\alpha},s_{cris,\alpha}),\eta\circ g)
\end{align}
Then $I(\mathbb{Q})$ acts on the left: for $h\in I(\mathbb{Q})$, it moves $((\rho:X_0\longrightarrow X),g)$ to $((\rho\circ h^{-1}:X_0\longrightarrow X), h\cdot g)$. $K^p$ just acts on the $G(\mathbb{A}_f^p)$ by the right multiplication. Taking the quotient one gets $\Theta_b$. 

\begin{remark}
Even though here I only restate the uniformization map for the reduced locus of the basic locus, Kim\cite{Kim} constructed the map for all the Newton types and values in all $\overline{k}$-schemes $S$ on which $p$ is locally nilpotent. 
\end{remark}

\section{Congruence relation in case SO$(V)$ is split at $p$}\label{split}
In this section I extend B\"ultel's results in his paper\cite{B} to more general cases.

The ordinary congruence relation is already known. The next question to ask is whether this is enough to deduce the conjecture\ref{conj}, i.e.whether this implies $H_p(F)=0$ in $\mathbb{Q}[p-Isog_k]$. 

As in section 3, $cl\circ\sigma(H(X))$ deletes the terms in $H(X)$ with non $\mu$-ordinary coefficients. But if there are no such terms, $H(F)=0$ in $\mathbb{Q}[p-Isog_k^{ord}]$ implies it is also $0$ in $\mathbb{Q}[p-Isog_k]$. 

In \cite{B}, B\"ultel showed how to prove this for certain orthogonal Shimura varieties. The idea is basically showing the projection maps $s$ and $t$ are finite away from the basic locus. This excludes the possibility of the finite height coefficients; For the basic locus, in general neither $s$ nor $t$ is finite. However, if the dimension of the basic locus $\mathscr{S}_b$ is small enough, one can still exclude the possibility of basic coefficients in $cl\circ\sigma(H(X))$. So one just needs to show the finiteness of the projection away from the basic locus and the smallness of the dimension of the basic locus. 

Look at the fiber of the projection map. Since the source and the target are symmetric, I only talk about the source projection. Let $p-Isog^{(c)}_k$ be the components on which multiplicatior is $p^c$, as defined in section 3. Consider a $\overline{k}$ point of $\mathscr{S}_k$, $x_0:\text{Spec}\overline{k}\longrightarrow\mathscr{S}_k$. Let the Newton stratum of the image of $x_0$ has type $[b]$. The fiber over $x_0$ is just the fiber product over $\mathscr{S}_k$
\begin{equation}
\xymatrix{\text{Spec}\overline{k}\times_{\mathscr{S}_k} p-Isog^{(c)}_k \ar[r] \ar[d] & p-Isog^{(c)}_k \ar[d] \\
\text{Spec}\overline{k} \ar[r] & \mathscr{S}_k}
\end{equation}
Consider the functor of point of Spec$\overline{k}\times_{\mathscr{S}_k}p-Isog^{(c)}_k$ on a $\overline{k}$-scheme $S$. Denote the abelian variety corresponding to $x_0$ by $(A_0,(s_{cris,\alpha,0}),\lambda_0,\eta_0)$. Then one sees that $(\text{Spec}\overline{k}\times_{\mathscr{S}_k}p-Isog^{(c)}_k)(S)$ is the following data
\begin{itemize}
\item An $S$-point of $\mathscr{S}_k$, this defines an abelian scheme $(A,(s_{cris,\alpha}),\lambda,\eta)$ by pulling back the abelian scheme over $\mathscr{S}_k$.
\item A $p$-isogeny from $(A_0\times_kS, (s_{cris,\alpha,0}),\lambda_0,\eta_0)$, which is the trivial family over $S$ with fiber $A_0$, to $(A,(s_{cris,\alpha},\lambda,\eta)$, s.t.the multiplicator is $p^c$. And the level structure is preserved, i.e.$\eta=\eta_0\circ f^{-1}$. 
\end{itemize}
By taking their $p$-divisible groups, one gets a homomorphism from Spec$k\times_{\mathscr{S}_k}p-Isog_k^{(c)}$ to the reduced locus of the Rapoport-Zink space associated with $b$
\begin{align*}
\text{Spec}\overline{k}\times_{\mathscr{S}_k}p-Isog_k^{(c)}&\longrightarrow RZ^{red,(c)}_b\\
((A_0,(s_{cris,\alpha,0}),\lambda_0,\eta_0)\xrightarrow{f}(A,(s_{cris,\alpha}),\lambda,\eta))&\longmapsto(X_0\xrightarrow{f}X)
\end{align*}
This is fine because changing the base point $id:X_0\xrightarrow{\sim}X_0$ results in an automorphism of the Rapoport-Zink space. 
\begin{proposition}\label{inj}
This map is injective
\end{proposition}
\begin{proof}
The proof actually follows from the same reason in the construction of the uniformization map at the end of last section. Given $(X_0\xrightarrow{f}X)\in RZ_b^{red,(c)}(S)$. Then the kernel of this isogeny is a finite flat group scheme of $p$-power order. Then one can recover its preimage to be $A_0\longrightarrow A_0/ker(f)$ and transfer all the additional structures from $A_0$ to $A_0/ker(f)$. 
\end{proof}
As a corollary,
\begin{proposition}
For any GSpin Shimura varieties with hyperspecial level structure at $p$. regardless of the behavior of the orthogonal group over $\mathbb{Q}_p$, there is no cycle which is generically finite height appearing in the coefficients of $H(X)$. 
\end{proposition}
\begin{proof}
As mentioned in section 3, since the coefficients of $H(X)$ all come from the specialization of the generic fiber, the are all of dimension $d = dim\mathscr{S}_k$. Take such a component $Z$. Suppose $Z$ is finite height of Newton type $[b]$. By proposition \ref{Sstrat}, $s(Z)$ is contained in $\cup_{b'\leq b}\mathscr{S}_{b'}$, whose dimension is strictly less than $d$. From the section, when $b$ is not basic, $dimRZ^{red}_b=0$. Therefore by \ref{inj}, $dim\text{Spec}\overline{k}\times_{\mathscr{S}_k}p-Isog_k^{(c)}=0$ over the non-basic locus. From the properness of the projection map $s$, $s$ is finite away from the basic locus. Therefore $dimZ\leq dim(\cup_{b'\leq b}\mathscr{S}_{b'})<dim\mathscr{S}_k=dimZ$ a contradiction. So such $Z$ does not exist. 
\end{proof}
\begin{theorem}
Let $V$ be a quadratic space of rank $N$ over $\mathbb{Q}$ such that its signature over $\mathbb{R}$ is $(N-2,2)$. Let $p$ be prime where $SO(V)$ is nonramified. When $N$ is odd or when $N$ is even and SO$(V)$ splits over $p$, the conjecture\ref{conj} holds.
\end{theorem}
\begin{proof}
The (generically) finite height coefficients in $\sigma(H(X))$. Take an irreducible component $Z$ of the coefficients of $\sigma(H(X))$, which is basic. From \ref{Sstrat}, $s(Z)$ is contained entirely in the basic locus. On the other hand, by \ref{inj} and \ref{unif}, the fiber of $s$ over $s(Z)$ is strictly less than half of the dimension of $\mathscr{S}_k$. Therefore $dimZ\leq dim\mathscr{S}_k=dimZ$, a contradiction.
\end{proof}

But there is one more case, i.e., the rank of $V$ is even and GSpin$(V_{\mathbb{Q}_p})$ quasi-split but nont split, it is still impossible to exclude the opportunity of the basic component in the coefficients of $\sigma(H(X))$. The explicit form of the Hecke polynomial is needed. 

\section{Congruence relation in case SO$(V)$ is quasi-split but not split at $p$}\label{quasi}
As seen in theorem\ref{unif}, when the orthogonal group SO$(V)$, or equivalently GSpin$(V)$, does not split at $p$, the dimension of the basic locus is exactly half of the dimension of the Shimura variety, so the simple dimension bounding argument in section\ref{shim} cannot exclude the possibility of the basic components in the coefficients of $\sigma(H(X))$. As a result, the ordinary congruence relation does not immediately imply the full version of the conjecture\ref{conj}. 

However, in\cite{K}, Koskivirta factored the Hecke polynomial so that one special factor 'kills' the basic cycles. I will show that the same phenomenon happens to quasi-split GSpin. 

Since I will only care about the basic cycles, $\mathscr{S}_{k,b}$ means the basic locus from now on. 

\subsection{Review of the root datum of GSpin}
First look at the root datum of the even split GSpin. Let $V_{\mathbb{Q}_p}$ be the quadratic space over $\mathbb{Q}_p$ as in section 5 and $dimV_{\mathbb{Q}_p}=2n$. The root datum of GSpin can be obtained from that of SO. A reference is Asgari's thesis\cite{As}.

SO$(V_{\mathbb{Q}_p})$, Spin$(V_{\mathbb{Q}_p})$ and GSpin$(V_{\mathbb{Q}_p})$ fit into the following diagram
\begin{equation*}
\xymatrix{\text{Spin}(V_{\mathbb{Q}_p})\times \mathbb{G}_m \ar[r] \ar[d] & \text{Spin}(V_{\mathbb{Q}_p}) \ar[d] \\
\text{GSpin}(V_{\mathbb{Q}_p}) \ar[r] & \text{SO}(V_{\mathbb{Q}_p})}
\end{equation*}
Let me explain the arrows: The upper horizontal one is just the projection onto the first factor; The right vertical one is the double cover; The left vertical one is defined by $(g,s)\longmapsto g\cdot s$, this is actually a surjection of group schemes, which defines GSpin$(V_{\mathbb{Q}_p})$ as a quotient of Spin$\times\mathbb{G}_m$ whose kernel is $\mu_2$; The lower horizontal arrow is just the right arrow in the exact sequence 
\begin{equation}
1\longrightarrow\mathbb{G}_m\longrightarrow\text{GSpin}(V_{\mathbb{Q}_p})\longrightarrow\text{SO}(V_{\mathbb{Q}_p})\longrightarrow 1
\end{equation}

Taking the matrix of the symmetric bilinear form of the quadratic space $V_{\mathbb{Q}_p}$ to be anti-diagonal, then a split torus $T$ can be chosen 
\begin{equation*}
T=\{diag(t_1,t_2,\dots,t_n,t_n^{-1},\dots,t_2^{-1},t_1^{-1}\}
\end{equation*}
and also a basis of the character lattice, and its dual basis of the cocharacter lattice
\begin{align*}
\chi_i:T\longrightarrow\mathbb{G}_m:diag(t_1,t_2,\dots,t_n,t_n^{-1},\dots,t_2^{-1},t_1^{-1})\longmapsto t_i\\
\chi_i^\lor:\mathbb{G}_m\longrightarrow T:t\longmapsto diag(1,\dots,t,1,\dots\dots,1,t^{-1},\dots,1)
\end{align*}
In this basis and its dual basis, a set of positive simple roots and coroots of SO is given by
\begin{align*}
R&=\{\alpha_1=\chi_1-\chi_2,\dots,\alpha_{n-1}=\chi_{n-1}-\chi_n,\alpha_n=\chi_{n-1}+\chi_n\}\\
R^\lor&=\{\alpha_1^\lor=\chi_1^\lor-\chi_2^\lor,\dots,\alpha_{n-1}^\lor=\chi_{n-1}^\lor-\chi_n^\lor,\alpha_n^\lor=\chi_{n-1}^\lor+\chi_n^\lor\}
\end{align*}
Let $\mathbb{Q}_{p^2}$ be the unique nonramified quadratic extension of $\mathbb{Q}_p$. To define the quasi-split outer form of SO$(V_{\mathbb{Q}_p})$, just take the Gal$(\mathbb{Q}_{p^2}/\mathbb{Q})$ action on the character and cocharacter lattice to be
\begin{equation*}
\sigma:\chi_n\longmapsto-\chi_n,\chi_n^\lor\longmapsto-\chi_n^\lor
\end{equation*}
for $\sigma$ the nontrivial element.

Take $\tilde{T}$ be a maximal torus of Spin surjecting to $T$. Then $\tilde{T}\times\mathbb{G}_m$ is a maximal split torus of Spin$(V_{\mathbb{Q}_p})\times\mathbb{G}_m$. Similarly take $T_\Delta$ to be the image of $\tilde{T}\times\mathbb{G}_m$ in GSpin which surjects to $T$
\begin{equation*}
\xymatrix{\tilde{T}\times\mathbb{G}_m \ar[r] \ar[d] & \tilde{T} \ar[d] \\
T_\Delta \ar[r] & T}
\end{equation*}
It is a maximal split torus of GSpin. Now taking the character lattices of the maximal tori corresponding to the diagram relating SO, Spin and GSpin, one gets the following diagram of lattices
\begin{equation*}
\xymatrix{X^*(\tilde{T})\times\mathbb{Z} & \ar[l] X^*(\tilde{T}) \\
\ar[u] X^*(T_\Delta) & \ar[u] \ar[l] X^*(T)}
\end{equation*}
Via the right vertical arrow, identify $X^*(\tilde{T})$ as a superlattice of $X^*(T)$ inside $X^*(T)_\mathbb{Q}$. In terms of the chosen basis, this lattice is given by
\begin{equation*}
X^*(\tilde{T})=\mathbb{Z}\chi_1+\dots+\mathbb{Z}\chi_n+\mathbb{Z}\cdot\frac{1}{2}(\chi_1+\dots+\chi_n)
\end{equation*}
Then $X^*\times\mathbb{Z}$ can be written as
\begin{equation*}
X^*(\tilde{T})=\mathbb{Z}\chi_1+\dots+\mathbb{Z}\chi_n+\mathbb{Z}\cdot\frac{1}{2}(\chi_1+\dots+\chi_n)+\mathbb{Z}\chi_0
\end{equation*}
Where $\chi_0$ corresponds to the projection to the $\mathbb{G}_m$ in the product Spin$(V_{\mathbb{Q}_p})\times\mathbb{G}_m$. This is because Spin$(V)$ is simply connected, so its coroot lattice coincide with its cocharacter lattice. Therefore its character lattice can be identified with the weight lattice of SO$(V_{\mathbb{Q}_p})$, which is generated by the cocharacter lattice of SO$(V_{\mathbb{Q}_p})$ together with the fundamental weight defining the half spin representation. 

Next let's try to identify the character lattice of GSpin as a sublattice of $X^*(\tilde{T})\times\mathbb{Z}$. According to Asgari\cite{As}, or Milne's book\cite{JM}, chapter 24, the left column of the first diagram fits into the following exact sequence
\begin{equation*}
1\longrightarrow\mu_2\longrightarrow\text{Spin}(V_{\mathbb{Q}_p})\times\mathbb{G}_m\longrightarrow\text{GSpin}(V_{\mathbb{Q}_p})\longrightarrow1
\end{equation*}
in which the second arrow is given by
\begin{equation*}
\mu_2\longrightarrow\mathbb{G}_m\xrightarrow{(\alpha^\lor_{n-1}+\alpha^\lor_n)\times id}\text{Spin}(V_{\mathbb{Q}_p})\times\mathbb{G}_m
\end{equation*}
So the character lattice of GSpin$(V_{\mathbb{Q}_p})$ as a sublattice of $X^*(\tilde{T})\times\mathbb{Z}$ are those vectors with integral values on $(1/2)\alpha_{n-1}^\lor+(1/2)\alpha_n^\lor$. This lattice is given by
\begin{equation*}
X^*(T_\Delta)=\mathbb{Z}\chi_1+\dots+\mathbb{Z}\chi_n+\mathbb{Z}\cdot(\chi_0+(1/2)\cdot(\chi_1+\cdot+\chi_n))
\end{equation*}
and its dual lattice is given by
\begin{equation*}
X_*(T_\Delta)=\mathbb{Z}\chi_0^\lor+\mathbb{Z}\cdot(\chi_0^\lor+(1/2)\cdot\chi_1^\lor)+\dots\dots+\mathbb{Z}\cdot(\chi_0^\lor+(1/2)\cdot\chi_n^\lor)
\end{equation*}
Like\cite{As}, in general people use another basis than $\chi$s. Let $e_0=\chi_0+(1/2)\cdot(\chi_1+\dots+\chi_n)$ and $e_i=\chi_i$. On the dual side, let $e_0^\lor=\chi_0^\lor$ and $e_i^\lor=\chi_i^\lor+(\chi_0^\lor)/2$. Rewrite the set of positive roots and coroots in this new basis
\begin{align*}
R&=\{\alpha_1=e_1-e_2,\dots,\alpha_n=e_{n-1}-e_n,\alpha_n=e_{n-1}+e_n\}\\
R^\lor&=\{\alpha_1^\lor=e_1^\lor-e_2^\lor,\dots,\alpha_{n-1}^\lor=e_{n-1}^\lor-e_n^\lor,\alpha_n^\lor=e_{n-1}^\lor+e_n^\lor-e_0^\lor\}
\end{align*}

Similar to SO$(V_{\mathbb{Q}_p})$, the quasi-split form of GSpin$(V_{\mathbb{Q}_p})$ is defined by the Galois action on the root datum
\begin{equation*}
\sigma:e_n\longmapsto e_0-e_n\text{ and }e_n^\lor\longmapsto e_0^\lor-e_n^\lor
\end{equation*}
in which $\sigma$ is the nontrivial involution of Gal$(\mathbb{Q}_{p^2}/\mathbb{Q})$. 

\subsection{The Hecke polynomial of quasi-split even GSpin Shimura varieties} 
In this section I compute the Hecke polynomial for $\mu^{KS}$. 
\begin{lemma}
In terms of the $e_i$s introduced right above, $\mu^{KS}=e_1^\lor$
\end{lemma}
\begin{proof}
Write $\mu^{KS}$ as a linear span of the $e$s first. Let $\mu^{KS}=x_0e_0^\lor+x_1e_1^\lor+\dots+x_ne_n^\lor$. Since it is the lift of $\mu=\chi_1^\lor$, it must pair zeroly with $\chi_2,\dots,\chi_n$. Therefore $x_2=x_3=\dots=x_n=0$ but $x_1=1$. Since $\eta(\mu^{KS}(t))=t$ and $\eta=2\chi_0^\lor,2x_0+x_1=1$. Hence $x_0=0$. So $\mu^{KS}$ is nothing but $e_1^\lor$. 
\end{proof}
\begin{lemma}
All the weights appearing in the representation of GSO determined by $\mu^{KS}$ are
\begin{equation*}
e^\lor_1,e^\lor_2,\dots,e^\lor_n;e_0^\lor-e_1^\lor,e_0^\lor-e_2^\lor,\dots,e_0^\lor-e_n^\lor
\end{equation*}
\end{lemma}
\begin{proof}
Since $\mu^{KS}$ is miniscule, the weights are simply the weights appearing in the Weyl group orbit of it. It is easy to check that these weights all appear in the Weyl group orbit of $\mu^{KS}$. Let $W_M$ be the Weyl group of the centralizer $M$ of $\mu^{KS}$. The length of the Weyl group orbit is
\begin{equation*}
\frac{|W|}{|W_M|}=\frac{2^n\cdot n!}{2^{n-1}\cdot(n-1)!}=2n
\end{equation*}
So these are just all the weights appearing in the representation of GSO determined by $\mu^{KS}$. 
\end{proof}

Since the Hecke polynomial is $\sigma$-conjugate invariant, I only need to look at the maximal torus
\begin{equation*}
H(X)=det(X-p^{n-1}\cdot r_{\mu^{KS}}(\sigma\ltimes t))
\end{equation*}
in which $t\in\hat{T}_\Delta$, the dual torus of $T_\Delta$ , which is a maximal torus of GSO. The maximal torus $T_\Delta$ is a rank $n+1$ torus. I choose a splitting $T_\Delta=\mathbb{G}_m\times\dots\dots\times\mathbb{G}_m$, in which $e_i^\lor$ corresponds to $t\longmapsto(1,\dots,1,t,1,\dots,1)$. One has the splitting of $\hat{T}_\Delta$ dual to the above splitting and write it down in the matrix form
\begin{equation*}
\hat{T}_\Delta=\{diag(s,t_1,\dots\dots,t_n)\}
\end{equation*}
Therefore, viewing $e_i^\lor$ as characters of $\hat{T}_\Delta$
\begin{equation}
e_0^\lor(diag(s,t_1,\dots\dots,t_n))=s\text{ and }e_i^\lor(diag(s,t_1,\dots\dots,t_n))=t_i
\end{equation}
The Kuga-Satake cocharacter defines the GSO-module with the highest weight $e_1^\lor$. As computed above, all of the weights are known, so when restricted to this maximal torus, $r_{\mu^{KS}}(t)$ is just the diagonal matrix
\[
\begin{bmatrix}
 t_1 & & & & &\\
& \ddots & & & &\\
& & t_n & & &\\
& & & st_n^{-1} & & &\\
& & & & \ddots &\\
& & & & & st_1^{-1}
\end{bmatrix}
\]
But I need the matrix of $r_{\mu^{KS}}(\sigma\ltimes t)$. The effect of $\sigma$ is changing the diagonal matrix a little bit, $r_{\mu^{KS}}(\sigma\ltimes t)$ is the following matrix, following Wedhorn\cite{W},2.7.1
\[
\begin{bmatrix}
p^{1-n}\cdot t_1 & & & & & &\\
& p^{2-n}\cdot t_2 & & & & &\\
& & \ddots & & & &\\
& & & & t_n & &\\
& & & st_n^{-1} & & &\\
& & & & \ddots & &\\
& & & & & p^{2-n}\cdot st_2^{-1} & \\
& & & & & & p^{1-n}\cdot st_1^{-1}
\end{bmatrix}
\]
Therefore the Hecke polynomial is 
\begin{equation}\label{hkp}
H(X)=(X^2-p^{2(n-1)}\cdot s)(X-t_1)(X-st_1^{-1})\dots\dots(X-p^{n-2}t_{n-1})(X-p^{n-2}st_{n-1}^{-1})
\end{equation}
Since both the first and the second factor are invariant under $\sigma$, they can be viewed as polynomials with coefficients in the Hopf algebra of the maximal split torus of the quasi-split GSO. The first factor is pretty simple. This simple factor kills the basic cycles in the same way as in Koskivirta\cite{K}. Notice that the cocharacter of $T_\Delta$ corresponding to $s$ is just the cocharacter in the exact sequence
\begin{equation*}
1\longrightarrow\mathbb{G}_m\longrightarrow\text{GSpin}(V_{\mathbb{Q}_p})\longrightarrow\text{SO}(V_{\mathbb{Q}_p})\longrightarrow 1
\end{equation*}
So viewed as an element of $K_p\backslash\text{GSpin}(V_{\mathbb{Q}_p})(\mathbb{Q}_p)/K_p$, it is $K_ppK_p$.

\subsection{Ideas of the proof}\label{ideas} Denote the product of the terms in $H(X)$ other than $(X^2-p^{2(n-1)}\cdot s)$ by $R(X)$ for short. Let $R(X)=a_mX^m+a_{m-1}X^{m-1}+\dots+a_0$. Where $a_i$s are elements in $\mathbb{Q}[p-Isog_k]$, which all have dimension $d=dim\mathscr{S}_k$. Now for each $a_i$ write it as
\begin{equation*}
a_i=b_i+c_i
\end{equation*}
where $b_i$s are supported on the cycles of $p-Isog_k$ who are generically ordinary. In other words, $b_i\in cl(\mathbb{Q}[p-Isog_k^{ord}])$. The following lemma is needed
\begin{lemma}
The following equation holds
\begin{equation*}
\sigma{H(X)}=\sigma((X^2-p^{2(n-1)}\langle p\rangle)\cdot R(X))
\end{equation*}
That is, the factorization of $H(X)$ commutes with $\sigma$. 
\end{lemma}
\begin{proof}
Specializing $(X^2-p^{2(n-1)}\cdot\langle p\rangle)\cdot R(X)$ , one gets
\begin{align*}
&\sigma((X^2-p^{2(n-1)}\langle p\rangle)\cdot (a_mX^m+a_{m-1}X^{m-1}+\dots+a_0))\\
&=\sigma(a_mX^{m+2}+a_{m-1}X^{m+1}+(a_{m-2}+p^{2(n-1)}\langle p\rangle\cdot a_m)\cdot X^m+\dots\dots-p^{2(n-1)}\langle p\rangle\cdot a_0)
\end{align*}
So the key to prove the lemma is trying to prove $\sigma(\langle p\rangle\cdot a_i)=\sigma(\langle p\rangle)\sigma(a_i)$. First observe that one only needs to prove this equality holds for any irreducible component of $p-Isog_{\mathbb{Q}}$. Take one such component, say $C$. By definition, $\langle p\rangle\cdot C$ is the image of $\langle p\rangle\times_{t,Sh,s}C$ under the map
\begin{equation*}
p-Isog_\mathbb{Q}\times_{t,Sh,s}p-Isog_\mathbb{Q}\longrightarrow p-Isog_\mathbb{Q}
\end{equation*} 
Similarly, $\sigma(\langle p\rangle)\cdot\sigma(C)$ is the image of $\langle p\rangle\times_{t,\mathscr{S}_k,s}\sigma(C)$ under the above map on the special fiber. If one can prove $\sigma(\langle p\rangle\times_{t,Sh,s}C)=\langle p\rangle\times_{t,\mathscr{S}_k,s}\sigma(C)$ as cycles of $p-Isog_k\times_{t,\mathscr{S}_k,s}p-Isog_k$, it is done. For this purpose, one only needs to check they have the same multiplicity on the same support. Consider the map
\begin{equation*}
p-Isog_\mathbb{Q}\times_{t,Sh,s}p-Isog_\mathbb{Q}\longrightarrow p-Isog_\mathbb{Q}
\end{equation*}
given by dropping the multiplication by $p$, i.e.
\begin{equation*}
(A_1\xrightarrow{\times p}A_1\xrightarrow{f}A_2)\longmapsto(A_1\xrightarrow{f}A_2)
\end{equation*}
This map admits a section given by reversing it. Taking specialization, the multiplicity of $\sigma(\langle p\rangle\times_{t,Sh,s}C)$ on its support is the same as the multiplicity of $\sigma(C)$ on its support in $p-Isog_k$. Using the same section on the special fiber proves that the multiplicity of $\langle p\rangle\times_{}\sigma(C)$ is the same as the multiplicity of $\sigma(C)$ on its support. Therefore, $\sigma(\langle p\rangle\times_{t,Sh,s}C)$ and $\langle p\rangle\times_{t,\mathscr{S}_k,s}\sigma(C)$ have the same multiplicity on its support respectively. Since their supports are both the image under the section of forgetting multiplication by $p$, they are the same cycle in $p-Isog_k\times_{t,\mathscr{S}_k,s}p-Isog_k$.
\end{proof}

With this lemma, I can proceed. One has
\begin{equation*}
cl\circ ord(H(X))=(X^2-p^{2(n-1)}\cdot\langle p\rangle)(b_mX^m+b_{m-1}X^{m-1}+\dots+b_0)
\end{equation*}
Then $H(X)$, viewed as a polynomial with coefficients in $\mathbb{Q}[p-Isog_k]$, can be written as the sum of two parts
\begin{equation}\label{twoparts}
H(X)=(X^2-p^{2(n-1)}\cdot\langle p\rangle)\cdot(c_mX^m+\dots+c_0)+cl\circ ord(H(X))
\end{equation}
According to the ordinary congruence relation, $cl\circ ord(H(X))=0$. So I only have to prove the vanishing of the first summand on $F$. For dimension reasons and the finiteness of the source projection over the non-basic locus, all the $c_i$s are supported on the basic cycles of $p-Isog_k$. Therefore it only needs to prove 
\begin{equation}\label{vanish}
(F^2-p^{2(n-1)}\cdot\langle p\rangle)\cdot C=0
\end{equation}
By Koskivirta\cite{K}, proposition 25, $F$ and $\langle p\rangle$ are in the center of $\mathbb{Q}[p-Isog_k]$. So to prove the lemma, one needs to show that $C\cdot F^2$ and $p^{2(n-1)}C\cdot\langle p\rangle$ have the same support and multiplicity in $\mathbb{Q}[p-Isog_k]$. 

To check they have the same support, first assume that 
\begin{equation*}
p-Isog_{k,b}\xrightarrow{(s,t)}\mathscr{S}_{k,b}\times\mathscr{S}_{k,b}
\end{equation*}
is a closed immersion where I use $p-Isog_{k,b}$ and $\mathscr{S}_{k,b}$ for the basic locus. In this case any $C$ where $dimC=dim\mathscr{S}_k$ is bijective onto its image. So for any two such components $C_1$ and $C_2$, to check that they are the same, one just has to show that $s(C_1)=s(C_2)$ and $t(C_1)=t(C_2)$, since $(s,t)(C_i)=s(C_i)\times_kt(C_i)$ for $i=1,2$ by dimension considerations.
\begin{proposition}\label{sandt}
The condition that $(s,t)$  is a closed immersion can always be achieved by taking the level structure small enough.
\end{proposition}
Look at Koskivirta\cite{K}, theorem 19. He also proved that once $C\cdot F^2=p^{2(n-1)}C\cdot\langle p\rangle$ for the case where $K^p$ is small enough, the general cases can also be deduced. Let $K'=K_pK'^p\subset K=K_pK^p$. Let $\mathscr{S}_k'$ and $\mathscr{S}_k$ be the special fiber for $Sh_{K'}$ and $Sh_K$. Also let $p-Isog'_k$ and $p-Isog_k$ be the $p$-isogeny spaces defined as above for $K'$ and $K$. There is an \'etale cover
\begin{align*}
\pi:\mathscr{S}'_k&\longrightarrow\mathscr{S}_k\\
\Pi:p-Isog_k'&\longrightarrow p-Isog_k
\end{align*}
which induces algebra homomorphisms
\begin{equation*}
\Pi:\mathbb{Q}[p-Isog_k']\longrightarrow\mathbb{Q}[p-Isog_k]
\end{equation*}
Koskivirta\cite{K}, lemma 27 proved
\begin{proposition}
There are quations
\begin{align*}
\Pi_*[\langle p\rangle]&=deg(\pi)[\langle p\rangle]\\
\Pi_*[F]&=deg(\pi)[F]
\end{align*}
So $H(F)=0$ in $\mathbb{Q}[p-Isog_k']$ implies $H(F)=0$ in $\mathbb{Q}[p-Isog_k]$.
\end{proposition}

It is enough to prove the conclusion \ref{vanish} over $\overline{k}$. Choose $K^p$ small enough so that $(s,t)$ is a closed immersion. Let $C$ be an irreducible component of $p-Isog_{\overline{k}}$. To compare $C\cdot F^2$ and $C\cdot\langle p\rangle$, look at their images under the source and target projections.
\begin{lemma}
Let $C$ be a basic irreducible component of $p-Isog_{\overline{k}}$. Then support of $s(C)$ and $t(C)$ are irreducible components $\mathscr{S}_{\overline{k},b}$. 
\end{lemma}
\begin{proof}
In the last section, the fiber of $s$ and $t$ can be embedded into the Rapoport-Zink space, which is of dimension $n-1$. Therefore the fibers $s$ and $t$ have dimension less than or equal to $n-1$, so $s(C)$ and $t(C)$ have dimension at least $n-1$. Since they are both contained in the basic locus of $\mathscr{S}_{\overline{k}}$. They must exactly have dimension $n-1$. So they are both irreducible components of $\mathscr{S}_{\overline{k},b}$. 
\end{proof}

Take any geometric point of $C\times_{t,\mathscr{S}_{\overline{k},b},s}\langle p\rangle$, it corresponds to a tuple
\begin{equation}\label{pt}
(A_1,(s_{cris,\alpha,1}),\lambda_1,\eta_1)\xrightarrow{f}(A_2,(s_{cris,\alpha,2}),\lambda_2,\eta_2)\xrightarrow{\times p}(A_2,(s_{cris,\alpha,2}),\lambda_2,p\cdot\eta_2)
\end{equation}
such that $f$ preserves the level structure. Its image under the source map is $(A_1,(s_{cris,\alpha,1}),\lambda_1,\eta_1)$, therefore $s(C\cdot\langle p\rangle)$ is just $s(C)$. Similarly, $s(C\cdot F^2)=s(C)$. So $s(C\cdot\langle p\rangle)=s(C\cdot F)$. 

Next I need to compare $t(C\cdot\langle p\rangle)$ and $t(C\cdot F^2)$. Take a geometric point \ref{pt} on $C\times_{t,\mathscr{S}_{\overline{k},b},s}\langle p\rangle$. Its image under the target projection is $(A_2,\dots,p\cdot\eta_2)$. Similarly, a geometric point on $(C\times_{t,\mathscr{S}_{\overline{k},b},s}F)\times_{t,\mathscr{S}_{\overline{k},b},s}F$ corresponds to a tuple
\begin{align*}
(A_1,(s_{cris,\alpha,1}),\lambda_1,\eta_1)&\xrightarrow{f}(A_2,(s_{cris,\alpha,2}),\lambda_2,\eta_2)\\
&\xrightarrow{F}(A_2^{(p)},(s^{(p)}_{cris,\alpha,2}),\lambda_2^{(p)},\eta_2^{(p)})\xrightarrow{F}(A_2^{(p^2)},(s^{(p^2)}_{cris,\alpha,2}),\lambda_2^{(p^2)},\eta_2^{(p^2)})
\end{align*}
So its image under the projection map is $(A_2^{(p^2)},(s^{(p^2)}_{cris,\alpha,2}),\lambda_2^{(p^2)},\eta_2^{(p^2)})$. There is just one question yet to be answered: Are $(A_2,(s_{cris,\alpha,2}),\lambda_2,\eta_2)$ and  $(A_2^{(p^2)},(s^{(p^2)}_{cris,\alpha,2}),\lambda_2^{(p^2)},\eta_2^{(p^2)})$ in the same irreducible component of $\mathscr{S}_{\overline{k},b}$? To compare them, make use of the Rapoport-Zink uniformization map \ref{unif}. The finer structure of the Rapoport-Zink space is also needed. 

\subsection{More about Rapoport-Zink uniformization, d'apres Howard-Pappas}
Recall the Rapoport-Zink uniformization map
\begin{equation}
\Theta:I(\mathbb{Q})\backslash RZ^{red}_b\times G(\mathbb{A}_f^p)/K^p\longrightarrow\mathscr{S}_{\overline{k},b}
\end{equation}
The strategy to prove $t(C\cdot\langle p\rangle)=t(C\cdot F^2)$ is: Pick up a geometric point $\mathscr{S}_{\overline{k},b}$ corresponding to $(A,(s_{cris,\alpha}),\lambda,\eta)$ such that it lies on a unique irreducible component of $\mathscr{S}_{\overline{k},b}$. Then find points on $RZ_b^{red}\times G(\mathbb{A}_f^p)$ sitting over $(A,(s_{cris,\alpha}),\lambda,p\cdot\eta)$ and $(A^{(p^2)},s_{cris,\alpha}^{(p^2)},\lambda^{(p^2)},\eta^{(p^2)})$ respectively, trying to prove that these two points should map to the same irreducible components under the Rapoport-Zink uniformization map. 

Take $x_0$ to be a $\overline{k}$ point of $\mathscr{S}_{\overline{k},b}$. Let $(A_0,(s_{cris,\alpha,0}),\lambda_0,\eta_0)$ be its corresponding abelian variety. Recall the construction of $\Theta$ in section (5.3). Fix an isogeny $\rho:A_0\longrightarrow A$. Let $(X_0\xrightarrow{\rho}X)$ be the induced map on their $p$-divisible groups. 
\begin{lemma}
With all these notations, we have
\begin{align}
\Theta: (X_0\xrightarrow{\rho}X, \eta^{-1}_0\circ\rho^{-1}\circ\eta)&\longmapsto(A,(s_{cris,\alpha}), \lambda,\eta)\label{eq}\\
\Theta: (X_0\xrightarrow{p\cdot\rho}X, \eta_0^{-1}\circ\rho^{-1}\circ\eta)&\longmapsto (A,(s_{cris,\alpha}), \lambda,p\cdot\eta)\label{eqp}\\
\Theta: (X_0\xrightarrow{Fr^2\circ\rho}X^{(p^2)}, \eta_0^{-1}\circ\rho^{-1}\circ\eta)&\longmapsto (A^{(p^2)},(s^{(p^2)}_{cris,\alpha}), \lambda^{(p^2)},\eta^{(p^2)})\label{eqf}
\end{align}
\end{lemma}
\begin{proof}
From the equation
\begin{equation}
V\otimes\mathbb{A}\xrightarrow{\eta_0}H_1(A_0,\mathbb{A}_f^p)\xrightarrow{\rho}H_1(A,\mathbb{A}_f^p)
\end{equation}
one knows that transferring the level structure of $A_0$ to $A$ via $\rho$ is $\rho\circ\eta_0$. From the construction of $\Theta$ in \ref{unif}, the level structure part on the left hand side of \ref{eq} maps to $\rho\circ\eta_0\circ\eta_0^{-1}\circ\rho^{-1}\circ\eta=\eta$. So \ref{eq} is proved.

Similarly, transferring the level structure of $A_0$ to $A$ via $p\cdot\rho$, it is $p\cdot\rho\eta_0$. The level structure map on the left hand side of \ref{eqp} maps under $\Theta$ to $p\cdot\rho\circ\eta_0\circ\eta^{-1}\circ\rho^{-1}\circ\eta=p\cdot\eta$. It agrees with the right hand side.

Finally, transfer $\eta_0$ to $A^{(p)}$ via $\rho$ composed with the relative Frobenius $Fr\circ\rho:A\longrightarrow A^{(p)}$, from the equation
\begin{equation}
V\otimes\mathbb{A}_f^p\xrightarrow{\eta_0}H_1(A_0,\mathbb{A}_f^p)\xrightarrow{\rho}H_1(A,\mathbb{A}_f^p)\xrightarrow{Fr}H_1(A^{(p)},\mathbb{A}_f^p)
\end{equation}
it is $Fr\circ\rho\circ\eta_0$. So under $\Theta$, it is mapped to $Fr\circ\rho\circ\eta_0\circ\eta_0^{-1}\circ\rho^{-1}\circ\eta=Fr\circ\eta$. Apply $Fr$ twice, it is $Fr^2\circ\eta_0$ on the level structure part, it agrees with the right hand side of \ref{eqf}.
\end{proof}

One just has to show that $(X_0\xrightarrow{p\cdot\rho}X,\eta_0\circ\rho^{-1}\circ\eta^{-1})$ and $(X_0\xrightarrow{Fr^2\circ\rho}X^{(p^2)},\eta_0\circ\rho^{-1}\circ\eta^{-1})$ maps to the same irreducible component of $\mathscr{S}_{\overline{k},b}$, for this purpose it is enough to check that $(X_0\xrightarrow{p\cdot\rho}X)$ and $(X_0\xrightarrow{Fr^2\circ\rho}X^{(p^2)})$ are in the same irreducible component of $RZ_b^{red}$. 

In\cite{HP}, Howard and Pappas described the structure of $RZ^{red}_b$ quite explicitly in terms of linear algebra. Recall in section\ref{gsd}, I defined $V$ and $D$. There is a twisted Frobenius endomorphism $F$ of $D_K=D\otimes K$ defined by $b\circ\sigma$. Since $V_K=V\otimes K\subset End(D_K)$, define a twisted Frobenius on $V_K$ by conjugation
\begin{equation*}
\Phi: V_K\longrightarrow V_K,\space f\longmapsto F\circ f\circ F^{-1}
\end{equation*}
From this, define the inner form of $V_{\mathbb{Q}_p}$ by
\begin{equation*}
V^\Phi_K=\{x\in V_K:\Phi(x)=x\}
\end{equation*}
The automorphism group of the basic isocrystal $J_b$ in theorem\ref{unif} is simply GSpin$(V^\Phi_K)$, and it sits inside an exact sequence
\begin{equation*}
1\longrightarrow\mathbb{G}_m\longrightarrow\text{GSpin}(V_K^\Phi)\longrightarrow\text{SO}(V^\Phi_K)\longrightarrow 1
\end{equation*}
In\cite{HP}(5.1.1), they have
\begin{definition}
A $\mathbb{Z}_p$ sublattice $\Lambda$ of $V_K^\Phi$ is called a vertex lattice, if
\begin{equation*}
p\Lambda\subset\Lambda^\lor\subset\Lambda
\end{equation*}
and the type of $\Lambda$ is $t_\Lambda=dim_{k}(\Lambda/\Lambda^\lor)$. A sublattice of $V_K$ is called a special lattice if 
\begin{equation*}
(L+\phi(L))/L\xrightarrow{\sim}W/pW
\end{equation*}
\end{definition}
Given $(\rho:X_0\longrightarrow X_y,(s_\alpha))$, $\rho$ defines a map of the contravariant Dieudonn\'e isocrystal $\rho:\mathbb{D}(X_y)[1/p]\longrightarrow\mathbb{D}(X_0)[1/p]=D_K$. Let $M_y\in D_K$ to be $\rho(\mathbb{D}(X_y))$ in $D_K$. Let $M_{1,y}$ be the lattice $F^{-1}(pM_y)=\rho(V\mathbb{D}(X))$. They defined three $W$-lattices in $V_K$ for $M_y$
\begin{align*}
L_y&=\{x\in V_K:xM_{1,y}\subset M_{1,y}\}\\
L_y^\#&=\{x\in V_K:xM_y\subset M_y\}\\
L_y^{\#\#}&=\{x\in V_K:xM_{1,y}\subset M_y\}
\end{align*}
They satisfy the relation $\Phi(L_y)=L_y^\#$ and $L_y+L_y^\#=L_y^{\#\#}$. Note that $M_y$ and $pM_y$ share all these $L$s. They proved the following\cite{HP}, 6.2.2
\begin{proposition}
There exists a bijection of sets
\begin{align}\label{RZtol}
p^{\mathbb{Z}}\backslash RZ_b^{red}(\overline{k})&\longrightarrow\{\text{special lattices in }L\subset V_K\}\\
(X_y,\rho,(s_\alpha))&\longmapsto L_y
\end{align}
\end{proposition}
Therefore, they can describe $p^\mathbb{Z}\backslash RZ_b^{red}$ in terms of the special lattices in $V_K$. 

They attached a classical Deligne-Lusztig variety $S_\Lambda$ to each vertex lattice $\Lambda$  as follows: Let $\Omega_0=\Lambda/\Lambda^\lor$ be the $t_\Lambda$ dimensional vector space over $k$. Use $Q$ to denote the quadratic form on $V^\Phi_K$. Then $pQ$ is integral valued on $\Lambda$, it makes $\Omega_0$ into a nondegenerated quadratic space over $k$. Define $OGr(\Omega)$ to be the $k$ scheme whose functor of point on a $k$-ring $R$ is
\begin{equation*}
\{\text{totally isotropic local direct summands }\mathscr{L}\subset\Omega\otimes_kR\text{ of dimension }t_\Lambda/2\}
\end{equation*}
Since $OGr(\Omega)$ is defined over $k$, $OGr(\Omega)_{\overline{k}}$ has relative Frobenius endomorphism over $\overline{k}$
\begin{equation*}
\Phi:OGr(\Omega)_{\overline{k}}\longrightarrow OGr(\Omega)^{(p)}_{\overline{k}}=OGr(\Omega)_{\overline{k}}
\end{equation*}
Let $S_\Lambda$ to be the subscheme of $OGr(\Omega)_{\overline{k}}$ whose functor of point on a $\overline{k}$-algebra is given by
\begin{equation*}
S_\Lambda(R)=\{\mathscr{L}\subset OGr(\Omega)(R):dim(\mathscr{L}+\Phi(\mathscr{L}))=t_\Lambda/2 + 1\}
\end{equation*}
It is easy to see
\begin{proposition}
There is a bijection of sets
\begin{equation}\label{ltoS}
\{\text{special lattices }L\subset V_K:\Lambda_W^\lor\subset L\subset\Lambda_W\}\xrightarrow{\sim}S_\Lambda(\overline{k})
\end{equation}
which is $\Phi$-equivalent.
\end{proposition}
The following property of $S_\Lambda$ is known\cite{HP} 5.3.2
\begin{proposition}
$S_\Lambda$ has two connected components $S_\Lambda=S_\Lambda^+\sqcup S_\Lambda^-$ and the Frobenius $\Phi$ interchanges these two components.
\end{proposition}

For each special lattice $L$, there exists a unique smallest vertex lattice $\Lambda$ s.t. $\Lambda_W^\lor\subset L\subset\Lambda_W$. $\Lambda$ can be found in the following way. Define
\begin{equation}\label{iter}
L^{(r)}=L+\Phi(L)+\dots+\Phi^r(L)
\end{equation}
There exists a smallest $d$ s.t. $L^{(d)}=L^{(d+1)}$. Then $L^{(d)}$ descents to a vertex lattice $\Lambda(L)$ in $V_K^\Phi$. 

Define $RZ^{red}_\Lambda\subset RZ_b^{red}$ as the closed formal subscheme define by the condition
\begin{equation*}
\rho\circ\Lambda^\lor\circ\rho^{-1}\subset End(X_0)
\end{equation*}
Then $RZ_b^{red}=\cup_{\Lambda}RZ^{red}_\Lambda$ and $RZ^{red}_\Lambda(\overline{k})$ consists of the special lattices sitting between $\Lambda_W^\lor$ and $\Lambda_W$. For two lattices, $\Lambda$ and $\Lambda'$, $RZ^{red}_\Lambda$ intersects $RZ^{red}_{\Lambda'}$ if and only if $\Lambda\cap\Lambda'$ is again a vertex lattice, then $RZ^{red}_\Lambda\cap RZ^{red}_{\Lambda'}=RZ^{red}_{\Lambda\cap\Lambda'}$.

Howard-Pappas proved the following fact\cite{HP},6.3.1
\begin{theorem}
There is a unique isomorphism of $\overline{k}$-schemes
\begin{equation*}
p^\mathbb{Z}\backslash RZ_\Lambda^{red}\xrightarrow{\sim}S_\Lambda
\end{equation*}
such that on the level of $\overline{k}$-points, it is the composition of the maps \ref{RZtol} and \ref{ltoS}.
\end{theorem}
Also, fixing a $\Lambda$, there is a decomposition $RZ^{red}_\Lambda=RZ^{red,(odd)}_\Lambda\sqcup RZ^{red,(even)}_\Lambda$, in which $RZ^{red,(odd)}$ means the locus with odd multiplicator, $RZ^{red,(even)}$ means the locus with even multiplicator. In\cite{HP},6.3.2
\begin{proposition}
There is an isomorphism of schemes over $\overline{k}$
\begin{equation}\label{decomp}
p^\mathbb{Z}\backslash RZ^{red,(odd)}_\Lambda\sqcup p^\mathbb{Z}\backslash RZ^{red,(even)}_\Lambda\xrightarrow{\sim}S_\Lambda^+\sqcup S_\Lambda^-
\end{equation}
which is Frobenius equivariant.
\end{proposition}
More precisely, this means $RZ^{red,(odd)}_\Lambda$ maps to one of $S_\Lambda^{\pm}$ and $RZ_\Lambda^{red,(even)}$ maps to the other. The irreducible components of $RZ^{red}$ are exactly all $RZ^{red,(l)}_\Lambda$ with $\Lambda$ the largest type for all integers $l$. 

\subsection{Comparison of the support of $t(C\cdot\langle p\rangle)$ and $(C\cdot F^2)$} Let's see how multiplication by $p$ and $Fr$ move the irreducible components of $RZ_b^{red}$.

Take an irreducible component $RZ_\Lambda^{red,(l)}$, i.e. fix an integer $l$ and a lattice $\Lambda$ of largest type. To see how multiplication by $p$ moves this irreducible component, one just has to see how it moves a general $\overline{k}$-valued point $y$ on $RZ_\Lambda^{red,(l)}$ where a general point means a point lying on only one irreducible component rather than lying on the intersection of many components. From the above description by Howard-Pappas, this means the only vertex lattice whose $W$-span contains the special lattice corresponding to $y$ is $\Lambda$. 
\begin{lemma}
Let $(X_0\xrightarrow{\rho}X_y)$ be the corresponding quasi-isogeny of p-divisible groups for $y$. Then the point corresponding to $(X_0\xrightarrow{p\cdot\rho}X_y)$ is a general point of $RZ_\Lambda^{red,(l+2)}$.
\end{lemma}
\begin{proof}
From the definition of $L_y$ in \ref{eq}, $(X_0\xrightarrow{\rho}X_y)$ and $(X_0\xrightarrow{p\cdot\rho}X_y$ map to the same special lattice. Therefore they both map into $p^\mathbb{Z}\backslash RZ_\Lambda^{red}$ for the unique $\Lambda$, by taking their special lattice. Since multiplying $p$ increases the multiplicator by $p^2$, $(X_0\xrightarrow{p\cdot\rho}X_y^{(p)})\in RZ_\Lambda^{red,(l+2)}$. 
\end{proof}
\begin{lemma}
Let $y$ be as above, then the point corresponding to $(X_0\xrightarrow{Fr\circ\rho}X_y^{(p)})$ is a general point of $RZ^{red,(l+1)}_\Lambda$. Therefore, the point corresponding to $(X_0\xrightarrow{Fr^2\circ\rho}X_y^{(p^2)})$ is a general point of $RZ^{red,(l+2)}_\Lambda$.
\end{lemma}
\begin{proof}
For $(X_0\xrightarrow{Fr\circ\rho}X_y^{(p)})$, the induced map on the Dieudonn\'e isocrystals maps $\mathbb{D}(X_y^{(p)})[1/p]$ to the lattice is $\Phi(L_y)$ in $V_K$. According to the discussion above \ref{iter}, the vertex lattice $\Lambda$ is characterized by
\begin{equation}
\Lambda_W=L^{(d)}_y=L_y+\Phi(L_y)+\dots+\Phi^d(L_y)
\end{equation}
for $d$ the maximal type. So $\Phi(\Lambda_W)=\Lambda_W=\Phi(L_y)+\dots+\Phi^d(\Phi(L_y))$, i.e. $\Phi(L_y)$ is also contained in $\Lambda_W$. The lattice $\Lambda$ is the unique vertex lattice whose $W$-span contains $\Phi(L_y)$, because if $\Phi(L_y)$ is contained in the $W$-span of another vertex lattice $\Lambda'$, then $L_y$ is also contained in $\Lambda'_W$. Therefore, $(X_0\xrightarrow{Fr\circ\rho}X_y^{(p)})$ is also contained in $RZ^{red}_\Lambda$. More precisely, if $y$ maps to $S^\pm_\Lambda$, then $Fr(y)$ maps to $S^\mp_\Lambda$ under the map \ref{decomp}.

Since the multiplicator of the Frobenius is $p$, so $(X_0\xrightarrow{Fr\circ\rho}X^{(p)})\in RZ^{red,(l+1)}_\Lambda$. Apply the Frobenius twice, $(X_0\xrightarrow{Fr^2\circ\rho}X_y^{(p^2)})$ is a general point on $RZ_\Lambda^{red,(l+2)}$. 
\end{proof}
\begin{proposition}\label{supp}
Let $(X_0\xrightarrow{\rho}X)$ be as defined at the beginning of section 7.4, then $(X_0\xrightarrow{p\cdot\rho}X,\eta_0\circ\rho^{-1}\circ\eta^{-1})$ and $(X_0\xrightarrow{Fr^2\circ\rho}X^{(p^2)},\eta_0\circ\rho^{-1}\circ\eta^{-1})$ map to the same irreducible component of $\mathscr{S}_{\overline{k},b}$. Therefore $t(C\cdot\langle p\rangle)$ and $t(C\cdot F^2)$ have the same support in $\mathscr{S}_{k,b}$. So $C\cdot\langle p\rangle$ and $C\cdot F^2$ have the same support in $p-Isog_k$. 
\end{proposition}
\begin{proof}
Because they are on the same irreducible component of $RZ^{red}_b\times G(\mathbb{A}_f^p)/K^p$, which is the unique irreducible component containing both of them, from the two lemmas just above. Therefore they map to the same irreducible component of $\mathscr{S}_{\overline{k},b}$ under the uniformization map, and this is the unique irreducible component containing their image. This component is just $t(C\cdot\langle p\rangle)=t(C\cdot F^2)$. 
\end{proof}

\subsection{Comparison of multiplicity: proof of the conjecture\ref{conj}} Finally it is the time to prove the conjecture\ref{conj}.
\begin{theorem}
Let $H(X)$ be the Hecke polynomial \ref{hkp}, considered as a polynomial with coefficients in $\mathbb{Q}[p-Isog_k]$ via $\sigma\circ h$. Let $F$ be the Frobenius cycles as defined in section\ref{twosections}. Then $H(F)=0$.
\end{theorem}
\begin{proof}
To prove this theorem, one needs to compare the multiplicity of $C\cdot F^2$ and $C\cdot\langle p\rangle$ in $\mathbb{Q}[p-Isog_k]$. It is enough to compare them in $\mathbb{Q}[p-Isog_{\overline{k}}]$. For this purpose, follow the commutative diagram of Koskivirta\cite{K}, lemma 28

\[
\xymatrix{C\times_{t,s}F^2 \ar[r]^c & X\ar[r]^-{(s,t)} & C_s\times \mathscr{F}^2(C_t)\\
C \ar[u]^{\simeq} \ar[ur]^{c_F} \ar[r]^{(s,t)} \ar[dr]^{c_p} \ar[d]_{\simeq} & C_s\times C_t \ar[ur]^{id\times\mathscr{F}^2} \ar[dr]^{id\times\langle p\rangle}\\
C\times_{t,s}\langle p \rangle \ar[r]^{c} & X \ar[r]^-{(s,t)} & C_s\times\mathscr{F}^2(C_t)}
\]
In the diagram, $C_s$ is the support of the image $s(C)$ in $\mathscr{S}_{\overline{k},b}$. Similar for $C_t$. The horizontal arrows $c$ means the same as in \ref{mult}, i.e.the morphism:$p-Isog_{\overline{k}}\times_{t,\mathscr{S}_{\overline{k}},s}p-Isog_{\overline{k}}\longrightarrow p-Isog_{\overline{k}}$. In the diagram the $X$ means the support of $c_*(C\times_{t,s}F^2)$ in $p-Isog_{\overline{k}}$. The $(s,t)$ is the same as in \ref{sandt}:$p-Isog_{\overline{k}}\longrightarrow\mathscr{S}_{\overline{k},b}\times\mathscr{S}_{\overline{k},b}$. The $\mathscr{F}$ is the relative Frobenius on $C_t$. The map $\langle p\rangle:C_t\longrightarrow\mathscr{F}^2(C_t)$ just means $C_t\longrightarrow p-Isog_{\overline{k}}\xrightarrow{t}\mathscr{S}_{\overline{k},b}$ in which the first arrow is the section defining $\langle p\rangle$. This is well defined since $\mathscr{F}(C_t)=t(C\cdot F^2)=t(C\cdot\langle p\rangle)$ following \ref{supp}.

Let's compare the degrees of the two horizontal $c$s in the above diagram. Since the two left vertical arrows are isomorphisms, it just needs to compare the degree of $c_F$ and $c_p$. From the right part of the commutative diagram, this is reduced to compare the degrees of the map $\langle p\rangle:C_t\longrightarrow\mathscr{F}^2(C_t)$ and the square of the relative Frobenius $\mathscr{F}^2$. The former one obviously has degree $1$. The latter one has degree $p^{2(n-1)}$, since the relative Frobenius $\mathscr{F}$ has degree $p^{dim\mathscr{S}_{k,b}}$ and in this case $dim\mathscr{S}_{k,b}=n-1$. Therefore $C\cdot F^2=c_*(C\times_{t,s}F^2)=p^{2(n-1)}c_*(C\times_{t,s}\langle p\rangle)=p^{2(n-1)}\cdot C\cdot\langle p\rangle$. In other words, $C\cdot(F^2-p^{2(n-1)}\langle p\rangle)=0$. 

Recall the factorization of the Hecke polynomial \ref{hkp} and in \ref{ideas}, I wrote the expansion of the factors except $(X^2-p^{2(n-1)}\cdot\langle p\rangle)$ as
\begin{equation*}
\Sigma_{i=0}^mb_i\cdot X^i + \Sigma_{i=0}^mc_i\cdot X^i
\end{equation*}
where $c_i$s are supported on the basic cycles of $p-Isog_k$. From the fact just proved in the last paragraph, $(F^2-p^{2(n-1)}\cdot\langle p\rangle)\cdot(\Sigma_{i=0}^mF^i)=0$. Recall \ref{twoparts}, combining this fact with the ordinary congruence relation, $H(F)=0$. 
\end{proof}

\bibliographystyle{plain}
\bibliography{bib-2}

\end{document}